\documentclass[12pt,a4paper]{article}
\usepackage{lmodern}
\usepackage{float}
\usepackage[utf8]{inputenc}
\usepackage[T1]{fontenc}
\usepackage[spanish,es-tabla,es-nodecimaldot,shorthands=off]{babel}

\usepackage{amsmath,amssymb,amsthm,amsfonts}
\newcommand{\C}{\mathbb{C}}
\usepackage{mathtools}
\usepackage{mathrsfs}
\usepackage{stmaryrd}

\usepackage{tikz}
\usepackage{tikz-cd}
\usetikzlibrary{arrows,positioning,shapes,trees,calc,arrows.meta}

\usepackage{graphicx}

\usepackage[margin=2.5cm]{geometry}
\usepackage{enumitem}
\usepackage{hyperref}
\usepackage{xcolor}
\usepackage{titlesec}
\usepackage{fancyhdr}
\usepackage{orcidlink}

\hypersetup{
    colorlinks = true,
    linkcolor  = blue!70!black,
    citecolor  = red!60!black,
    urlcolor   = blue!70!black
}

\theoremstyle{plain}
\newtheorem{theorem}{Teorema}[section]
\newtheorem{proposition}[theorem]{Proposición}
\newtheorem{lemma}[theorem]{Lema}
\newtheorem{corollary}[theorem]{Corolario}

\theoremstyle{definition}
\newtheorem{definition}[theorem]{Definición}
\newtheorem{example}[theorem]{Ejemplo}

\theoremstyle{remark}
\newtheorem{remark}[theorem]{Observación}
\newtheorem{exercise}[theorem]{Ejercicio}

\newcommand{\Bool}{\mathbf{Bool}}
\newcommand{\Stone}{\mathbf{Stone}}
\newcommand{\ProFin}{\mathbf{ProFin}}
\newcommand{\Cond}{\mathbf{Cond}}
\newcommand{\Set}{\mathbf{Set}}
\newcommand{\Top}{\mathbf{Top}}
\newcommand{\Clop}{\operatorname{Clop}}
\newcommand{\Sub}{\operatorname{Sub}}
\newcommand{\Spec}{\operatorname{Spec}}
\newcommand{\Gal}{\operatorname{Gal}}
\newcommand{\Zp}{\mathbb{Z}_p}

\newcommand{\Cont}{\operatorname{Cont}}
\newcommand{\Hom}{\operatorname{Hom}}
\newcommand{\CHaus}{\mathbf{CHaus}}
\newcommand{\Z}{\mathbb{Z}}
\newcommand{\R}{\mathbb{R}}

\title{El mar que envuelve a las piedras\\[0.2cm]
\large \textit{Espacios de Stone, profinitos y su papel en la aritmética contemporánea}}

\author{J.~R. Pérez-Buendía\orcidlink{0000-0002-7739-4779}\\[0.1cm]
\small SECIHTI--CIMAT, Unidad Mérida\\
\small \texttt{rogelio.perez@cimat.mx}}
\date{\today}

\begin{document}

\maketitle

\begin{abstract}
\noindent
Este artículo de revisión y exposición presenta una descripción sistemática de
la dualidad de Stone, los espacios profinitos y la compactificación de
Stone--Čech, desde una perspectiva aritmética moderna que conecta estos conceptos
clásicos con las Matemáticas Condensadas de Clausen--Scholze.

Muchos objetos fundamentales en aritmética y geometría aritmética ---como los
enteros $p$-ádicos, los grupos de Galois absolutos y las variedades algebraicas
sobre campos de números--- poseen una estructura que combina aspectos discretos y
continuos. La dualidad de Stone establece una anti\-equivalencia entre la
categoría de álgebras de Boole y la categoría de espacios compactos, Hausdorff y
totalmente desconectados (espacios de Stone). Dentro de esta clase, los espacios
profinitos —aquellos que se describen como límites inversos de espacios finitos—
desempeñan un papel central en aritmética, pues modelan de forma natural estos
objetos aritméticos fundamentales.

El propósito de este artículo es describir de manera sistemática esta relación
desde un punto de vista aritmético: construimos el espacio de Stone asociado a un
álgebra de Boole, revisamos el teorema de representación de Stone y explicamos
cómo los espacios profinitos surgen como límites inversos de sistemas finitos.
Introducimos además la compactificación de Stone--Čech a partir de su propiedad
universal y discutimos en detalle ejemplos aritméticos relevantes (enteros
$p$-ádicos, grupos de Galois, aplicaciones en geometría aritmética). Finalmente,
mostramos cómo estos conceptos clásicos encuentran su lugar natural en el marco
de las Matemáticas Condensadas, donde los profinitos juegan el papel de
generadores fundamentales de la categoría de conjuntos condensados.

El texto está dirigido principalmente a estudiantes de posgrado en teoría de números, geometría aritmética o álgebra, que posean conocimientos básicos de topología general y teoría de categorías. Las primeras secciones (hasta antes de la sección de Matemáticas Condensadas) pueden ser leídas por estudiantes de licenciatura motivados con una base sólida en topología. La sección final sobre Matemáticas Condensadas requiere familiaridad con límites inversos, topología de Grothendieck y nociones básicas de sitios y gavillas.

Se proporciona una exposición unificada de resultados clásicos desde una perspectiva moderna, reorganizando y reinterpretando el material establecido con miras a la aritmética contemporánea y sus aplicaciones en geometría aritmética. Para lectores que necesiten repasar conceptos de topología o teoría de categorías, recomendamos consultar \cite{Engelking1989} y \cite{MacLane1998} respectivamente.
\end{abstract}

\tableofcontents

\vspace{0.7cm}

\section{Introducción: la tensión entre lo discreto y lo continuo}\label{sec:intro}

\subsection{Motivación aritmética}

Muchos de los objetos fundamentales en aritmética y geometría aritmética poseen
una estructura que combina aspectos discretos y continuos de manera sutil. Los
enteros $p$-ádicos $\mathbb{Z}_p$, por ejemplo, pueden pensarse como números
"infinitos" que codifican información sobre congruencias módulo potencias de
$p$. Los grupos de Galois absolutos $G_K = \Gal(\overline{K}/K)$ de un campo
$K$ son grupos "infinitos" que controlan todas las extensiones algebraicas de
$K$, pero su estructura se revela mejor cuando se los estudia como límites de
grupos finitos. Las variedades aritméticas ---objetos geométricos definidos
sobre campos de números--- presentan fenómenos que requieren entender cómo la
topología interactúa con estructuras algebraicas discretas.

El problema central que motiva este artículo es el siguiente: cuando intentamos
aplicar las herramientas estándar de topología y álgebra homológica a estos
objetos aritméticos, encontramos dificultades técnicas fundamentales. La
categoría de grupos topológicos no es abeliana, lo que impide usar álgebra
homológica estándar. Las cohomologías requieren definiciones ad hoc y cuidados
técnicos constantes. Los productos tensoriales y completaciones fallan en su
compatibilidad con la topología. Estas dificultades no son meros obstáculos
técnicos: reflejan una incompatibilidad profunda entre la topología clásica
(que privilegia la continuidad) y la estructura aritmética (que es
fundamentalmente discreta y filtrada).

La teoría que desarrollaremos se apoya en resultados clásicos. La dualidad de Stone entre álgebras de
Boole y espacios compactos, Hausdorff y totalmente desconectados data de los
años treinta \cite{Stone1936}; la identificación de estos espacios con los profinitos está
plenamente establecida \cite{Johnstone1982}; y la compactificación de Stone--Čech se describe en
detalle en monografías estándar \cite{ComfortNegrepontis1974,Johnstone1982,RibesZalesskii2010}.
Precisamente por esa solidez, estos resultados constituyen un terreno
privilegiado para ser revisitados: permiten reorganizar la teoría, precisar el
diccionario álgebra--topología y reinterpretar las construcciones básicas desde
un punto de vista más estructural, adaptado a la aritmética contemporánea y a
las formulaciones categóricas modernas. Para exposiciones modernas de estos temas, véase también \cite{DaveyPriestley2002,Engelking1989}.

El contexto matemático en el que se inscriben hoy estas ideas es radicalmente
distinto al de hace cincuenta años. Uno de los lugares donde esa diferencia se
hace más visible es el programa de \emph{Matemáticas Condensadas} de
Clausen--Scholze \cite{ScholzeCondensed,ScholzeCondensed2,Scholze2019,Bhatt2022}, que surge de la necesidad de un nuevo formalismo para la geometría $p$-ádica y la cohomología aritmética \cite{Scholze2012}. En ese marco, los espacios profinitos
---y, en particular, los espacios de Stone--- dejan de ser ejemplos
periféricos: el sitio de profinitos se convierte en la base sobre la cual se
definen gavillas, se reconstruyen nociones analíticas y se formulan teorías
cohomológicas adaptadas a la aritmética \cite{FarguesScholze2021}. Este artículo presenta la dualidad de
Stone, la descripción profinita y la compactificación de Stone--Čech desde esa
perspectiva, sin pretender reproducir exhaustivamente todas las pruebas
clásicas.

Este artículo es el primero de una serie de artículos expositivos dedicados a
explorar las relaciones entre álgebra y topología desde una perspectiva aritmética
y categórica moderna. El proyecto completo tiene como objetivo revisar, desde una
mirada aritmética y a la luz del enfoque condensado, distintas correspondencias
entre estructuras algebraicas y sus contrapartes topológicas y lógicas. En
artículos posteriores de esta serie aparecerán dualidades más generales ---como la
dualidad de Priestley para retículos distributivos, la dualidad de Esakia para
álgebras de Heyting y su relación con la lógica intuicionista--- así como la teoría
de marcos, locales y topoi elementales, todas situadas en un mismo paisaje
categórico y aritmético. El presente artículo constituye la primera capa de esa
arquitectura: la capa booleana y profinita sobre la cual se apoyan las teorías más
amplias.

\subsection{El problema: topología clásica vs estructura aritmética}

El punto de fricción sigue siendo esencialmente el mismo: la tensión persistente
entre álgebra y topología. Cuando dotamos de topologías naturales a objetos
aritméticos infinitos (grupos de Galois, anillos $p$-ádicos, grupos de puntos
de variedades), la topología clásica genera patologías bien conocidas:
\begin{itemize}
  \item categorías de grupos topológicos (o módulos topológicos) que no son
  abelianas;
  \item cohomologías que requieren definiciones ad hoc y cuidados técnicos
  constantes;
  \item productos tensoriales y completaciones cuya compatibilidad con la
  topología es delicada o directamente fallida.
\end{itemize}

En muchos contextos aritméticos, la topología clásica que se impone de manera
natural ---analítica, diferencial o compleja--- no refleja la estructura finita
y filtrada presente en el trasfondo algebraico. Cuando estudiamos
una variedad algebraica sobre un campo de números, la topología compleja (si
trabajamos con $\mathbb{C}$) o la topología $p$-ádica (si trabajamos con
$\mathbb{Q}_p$) no capturan adecuadamente la estructura de reducción módulo
primos, que es fundamental en geometría aritmética. Esa topología puede ser
excesivamente homogénea: no separa niveles discretos ni
discrimina adecuadamente subdatos de congruencia. 

Para recuperar esa granularidad lógica, conviene trabajar con espacios
totalmente desconectados, como los profinitos o los espacios de Stone, donde
los clopens codifican directamente subestructuras finitas. Estos espacios
reflejan naturalmente la estructura aritmética: cada nivel de congruencia, cada
extensión finita de Galois, cada reducción módulo una potencia de $p$, se
codifica como un clopen en el espacio correspondiente.

Si se lleva la metáfora del título un poco más lejos, la topología clásica que
solemos poner a nuestros objetos aritméticos se parece más a un bloque de roca
relativamente lisa: el mar algebraico la rodea, pero no consigue anclar en ella
una estructura combinatoria suficientemente fina. La continuidad permite
demasiadas deformaciones y muchas construcciones ---productos tensoriales,
completaciones, cohomologías--- resbalan sobre esa superficie sin quedar del
todo fijadas en la geometría.

En cambio, los espacios de Stone ---es decir, los espacios profinitos vistos a
través de su álgebra booleana de clopens--- se parecen más a un conjunto de
piedras menudas y porosas, casi una grava lógica: el agua penetra en cada
intersticio, y lo que ocurre en el mar se registra como información discreta,
codificada en clopens y en límites inversos finitos. Cada poro corresponde a un
clopen, y cada clopen a un nivel finito del sistema inverso. Desde la
perspectiva de las Matemáticas Condensadas, ese mar deja de ser un continuo
amorfo para convertirse en un medio profinito y lógico en el que estos espacios
profinitos/Stone juegan el papel de objetos generadores: sobre ellos se prueban
definiciones, se miden cohomologías y se construye la estructura condensada.

Históricamente, esta tensión entre lo algebraico y lo topológico se ha
reorganizado alrededor de ciertos hitos conceptuales. En 1936, Marshall
H.~Stone observa que la lógica booleana posee una geometría interna: toda
álgebra de Boole es isomorfa al álgebra de subconjuntos clopen de un espacio
compacto totalmente desconectado \cite{Stone1936}. Décadas más tarde, la
escuela de Grothendieck sistematiza el fenómeno de las anti\-equivalencias
álgebra–geometría (anillos versus esquemas afines, marcos versus locales),
proporcionando el lenguaje categórico para entender dualidades como
\[
  \Bool \;\simeq\; \Stone^{\mathrm{op}}.
\]

Esta manera de trabajar es cercana a la filosofía que Grothendieck describe en
sus escritos autobiográficos \cite{GrothendieckRS}: en lugar de atacar cada problema en su forma
concreta, se amplía el marco conceptual hasta que las dificultades particulares
quedan absorbidas en una teoría más general. En sus propias palabras, se deja
subir una ``marea creciente'' de abstracción que, poco a poco, cubre los
escollos locales y los integra en un paisaje más amplio y natural. En el
contexto que nos ocupa, esa marea es la topología compacta y la perspectiva
categórica; las piedras son los objetos algebraicos sobre los que se hace
visible su efecto.

En la formulación de Clausen y Scholze \cite{ScholzeCondensed}, este movimiento
se lleva un paso más allá. Los espacios profinitos no son sólo ejemplos
privilegiados, sino generadores fundamentales de una nueva categoría de
conjuntos ``condensados'' en la que la información topológica y algebraica se
amalgama en un único objeto estructurado. En lugar de estudiar $\mathbb{Z}_p$ o
un grupo de Galois aislados, se les sitúa dentro de un entorno en el que los
mapas desde profinitos controlan la noción misma de continuidad y de sección
global. La abstracción deja de ser un lujo para convertirse en el marco mínimo
en el que ciertos fenómenos aritméticos se vuelven formulables.

Este artículo se sitúa, por tanto, en la interfaz entre tres capas históricas:
la dualidad booleana clásica, la teoría profinita que la prolonga en contextos
aritméticos y el enfoque condensado que la integra en un marco categórico más
amplio. Se trata de un artículo de revisión que, sin pretender ser un tratado de lógica, 
un manual de teoría de topos, o un texto completo sobre Matemáticas Condensadas, 
ofrece una revisión organizada de espacios de Stone, espacios profinitos y
compactificación de Stone--Čech, escrita con miras a la aritmética moderna y al
andamiaje categórico que subyace al enfoque condensado.

\section{Posets, retículos y álgebras de Boole}\label{sec:bool-algebra}

\subsection{Posets y retículos}\label{subsec:posets-reticulos}

\begin{definition}\label{def:poset}
Un \textit{conjunto parcialmente ordenado} (poset) es un par $(P,\leq)$ donde $\leq$ es una relación binaria en $P$ tal que, para todos $x,y,z\in P$:
\begin{enumerate}[label=\roman*)]
  \item $x\leq x$ (reflexividad),
  \item $x\leq y$ y $y\leq x$ implican $x=y$ (antisimetría),
  \item $x\leq y$ y $y\leq z$ implican $x\leq z$ (transitividad).
\end{enumerate}
\end{definition}

\begin{definition}\label{def:lattice}
Un \textit{retículo} es un poset $(L,\leq)$ donde para todo par $a,b\in L$ existen el supremo $a\vee b$ y el ínfimo $a\wedge b$. Equivalentemente, $L$ está dotado de operaciones binarias
\[
\wedge,\vee : L\times L\to L
\]
conmutativas y asociativas, que satisfacen las leyes de absorción
\begin{equation}\label{eq:absorcion}
a\wedge (a\vee b) = a,\qquad a\vee (a\wedge b) = a,\quad \forall a,b\in L.
\end{equation}
\end{definition}

En un retículo, el orden se recupera a partir de $\wedge$:

\begin{lemma}\label{lem:orden-por-meet}
Sea $(L,\wedge,\vee)$ un retículo. Para $a,b\in L$ se tiene
\begin{equation}\label{eq:orden-por-meet}
a\leq b \;\Longleftrightarrow\; a\wedge b = a.
\end{equation}
\end{lemma}

\begin{proof}
Si $a\leq b$, entonces $a$ es una cota inferior de $\{a,b\}$; por definición de $a\wedge b$ se tiene $a\wedge b\leq a$. Por otra parte, $a\wedge b\leq a$ y $a$ es una cota inferior, así que $a\leq a\wedge b$. Por antisimetría, $a\wedge b=a$.

Recíprocamente, si $a\wedge b=a$, como el ínfimo es siempre menor o igual que cada uno de los elementos, $a\wedge b\leq b$ implica $a\leq b$.
\end{proof}

\begin{definition}\label{def:reticulo-distributivo}
Un retículo $(L,\wedge,\vee)$ es \textit{distributivo} si para todos $a,b,c\in L$ se cumple
\begin{equation}\label{eq:distributividad}
a\wedge (b\vee c) = (a\wedge b)\vee (a\wedge c).
\end{equation}
\end{definition}

\subsection{Álgebras de Boole}\label{subsec:boolean-algebras}

\begin{definition}\label{def:boolean-algebra}
Una \textit{álgebra de Boole} es una 6-tupla $(B,\wedge,\vee,\neg,0,1)$ donde:
\begin{itemize}
  \item $(B,\wedge,\vee)$ es un retículo distributivo con elemento mínimo $0$ y máximo $1$;
  \item $\neg:B\to B$ es una operación unaria llamada \emph{complemento} que satisface, para todo $a\in B$,
  \begin{equation}\label{eq:boolean-complemento}
  a\wedge \neg a = 0,\qquad a\vee \neg a = 1.
  \end{equation}
\end{itemize}
\end{definition}

\begin{remark}\label{rem:boolean-complementado}
Toda álgebra de Boole es un retículo distributivo \emph{complementado}: para cada $a\in B$ existe un complemento $\neg a$ que satisface \eqref{eq:boolean-complemento}. Más aún, este complemento es \emph{único}: si $b\in B$ satisface $a\wedge b = 0$ y $a\vee b = 1$, entonces $b = \neg a$. Esto se deduce de las leyes distributivas y de las propiedades del complemento. En particular, el paso de retículo distributivo a álgebra de Boole no es meramente añadir una operación unaria: se requiere que todo elemento tenga un complemento único.
\end{remark}

\begin{example}\label{ex:boolean-examples}
\leavevmode
\begin{enumerate}[label=\alph*)]
  \item Para todo conjunto $X$, la potencia $\mathcal{P}(X)$ con
  \[
  A\wedge B = A\cap B,\quad A\vee B = A\cup B,\quad \neg A = X\setminus A,\quad 0=\varnothing,\quad 1=X,
  \]
  es una álgebra de Boole.
  \item Sea $R$ un anillo conmutativo unitario. El conjunto de idempotentes
  \[
  E(R) := \{e\in R : e^2=e\}
  \]
  es una álgebra de Boole con
  \begin{equation}\label{eq:boolean-idempotentes}
    e\wedge f := ef,\quad e\vee f := e+f-ef,\quad \neg e := 1-e.
  \end{equation}
  Esta álgebra codifica la desconexión de $\Spec(R)$: descomposiciones de $1$ en sumas ortogonales de idempotentes corresponden a descomposiciones de $\Spec(R)$ en componentes abiertas--cerradas.
\end{enumerate}
\end{example}

\begin{definition}\label{def:atomos}
Sea $B$ un álgebra de Boole. Un elemento $a\in B$ es un \textit{átomo} si $a\neq 0$ y siempre que $0\leq b\leq a$ se tiene $b=0$ o $b=a$.
\end{definition}

\begin{exercise}\label{ex:boolean-atoms}
Prueba que si $B$ es una álgebra de Boole finita, entonces es isomorfa a $\mathcal{P}(S)$, donde $S$ es el conjunto de átomos de $B$. Sugerencia: usa que todo elemento se escribe como unión finita de átomos distintos.
\end{exercise}

\begin{remark}\label{rem:granulos}
En una álgebra de Boole finita $B$, los átomos juegan el papel de "granos elementales" de información: si $S$ denota el conjunto de átomos, el isomorfismo $B \cong \mathcal{P}(S)$ identifica cada elemento de $B$ con un subconjunto de $S$. 

Dado un subconjunto finito $\mathcal{A}\subseteq B$, la subálgebra booleana finita generada por sus elementos tiene un conjunto finito de átomos, que llamaremos los \emph{gránulos asociados a $\mathcal{A}$}. Estos gránulos son, en la práctica, los fragmentos discretos que aparecen al observar el espacio de Stone $X_B$ con la resolución finita impuesta por $\mathcal{A}$: cada átomo corresponde a un clopen minimal de $X_B$ en esa escala, una especie de grano de arena lógico sobre el que se apoyan las aproximaciones finitas.
\end{remark}

\section{Filtros, ultrafiltros y puntos lógicos}\label{sec:filtros-ultrafiltros}

\subsection{Filtros y ultrafiltros}\label{subsec:filtros}

\begin{definition}\label{def:filtro}
Sea $B$ un álgebra de Boole. Un \textit{filtro} en $B$ es un subconjunto $F\subseteq B$ tal que:
\begin{enumerate}[label=\roman*)]
  \item $1\in F$ y $0\notin F$;
  \item si $a,b\in F$, entonces $a\wedge b\in F$;
  \item si $a\in F$ y $a\leq b$, entonces $b\in F$.
\end{enumerate}
\end{definition}

\begin{definition}\label{def:ultrafiltro}
Un \textit{ultrafiltro} en $B$ es un filtro propio $U\subsetneq B$ que es maximal con respecto a la inclusión.
\end{definition}

\begin{lemma}[Caracterización de ultrafiltros]\label{lem:caracterizacion-ultrafiltro}
Sea $U$ un filtro propio en una álgebra de Boole $B$. Entonces $U$ es un ultrafiltro si y sólo si para todo $a\in B$ se tiene
\begin{equation}\label{eq:ultrafiltro-decide}
\textit{exactamente uno de } a,\neg a \textit{ pertenece a }U.
\end{equation}
\end{lemma}

\begin{proof}
$(\Rightarrow)$ Si $U$ es ultrafiltro (filtro propio maximal), entonces por la maximalidad no pueden faltar ambos $a,\neg a$ en $U$ (de lo contrario, podríamos extender $U$ agregando el que falta). Por otra parte, la propiedad de filtro impide que ambos estén en $U$, pues si $a,\neg a\in U$, entonces $0=a\wedge \neg a\in U$, contradicción. Por tanto, exactamente uno de $a,\neg a$ pertenece a $U$.

$(\Leftarrow)$ Si $U$ satisface \eqref{eq:ultrafiltro-decide} y $F$ es un filtro propio con $U\subseteq F\subseteq B$, toma $a\in F\setminus U$. Por \eqref{eq:ultrafiltro-decide}, $\neg a\in U\subseteq F$, y entonces $0=a\wedge \neg a\in F$, contradicción. Por tanto $F=U$.
\end{proof}

\begin{remark}\label{rem:Zorn-ultrafiltro}
Usando el lema de Zorn y el axioma de elección se prueba que todo filtro propio se puede extender a un ultrafiltro. Este hecho es equivalente a formas débiles del lema de ultrafiltros.
\end{remark}

\subsection{Ultrafiltros y homomorfismos a \texorpdfstring{$2$}{2}}\label{subsec:ultrafiltros-hom}

Sea $2=\{0,1\}$ la álgebra de Boole de dos elementos.

\begin{proposition}\label{prop:ultrafiltros-hom-2}
Para una álgebra de Boole $B$, existe una biyección natural
\begin{equation}\label{eq:ultrafiltros-hom-2}
\{\textit{ultrafiltros en }B\}\;\cong\;\Hom_{\Bool}(B,2).
\end{equation}
Más aún, esta biyección es \emph{natural} en el sentido categórico: define una equivalencia de categorías entre la categoría de ultrafiltros (con morfismos inducidos por homomorfismos de álgebras de Boole) y la categoría de homomorfismos $B\to 2$.
\end{proposition}

\begin{proof}
Dado un ultrafiltro $U$, definimos $v_U:B\to 2$ por
\[
v_U(a)=
\begin{cases}
1,&\textit{si } a\in U,\\
0,&\textit{si } a\notin U.
\end{cases}
\]
El Lema~\ref{lem:caracterizacion-ultrafiltro} y las propiedades de filtro garantizan que $v_U$ respeta $\wedge,\vee,\neg,0,1$, es decir, es homomorfismo de álgebras de Boole.

Recíprocamente, un homomorfismo $v:B\to 2$ da un ultrafiltro
\[
U_v := v^{-1}(\{1\}).
\]
Se verifica directamente que $U_v$ es filtro y que satisface \eqref{eq:ultrafiltro-decide}, por lo que es ultrafiltro. Las aplicaciones $U\mapsto v_U$ y $v\mapsto U_v$ son inversas una de otra.

La naturalidad se verifica observando que si $\varphi:B\to B'$ es un homomorfismo de álgebras de Boole y $U'$ es un ultrafiltro en $B'$, entonces $\varphi^{-1}(U')$ es un ultrafiltro en $B$, y esta construcción es compatible con la correspondencia entre ultrafiltros y homomorfismos a $2$.
\end{proof}

\section{El espacio de Stone y el teorema de representación}\label{sec:stone-space}

\subsection{Construcción del espacio de Stone}\label{subsec:stone-construccion}

\begin{definition}\label{def:espacio-Stone-B}
Sea $B$ un álgebra de Boole. El \textit{espacio de Stone} de $B$ (también llamado la piedra asociada a $B$), es el conjunto
\begin{equation}\label{eq:def-XB}
X_B := \{U\subseteq B : U \textit{ es un ultrafiltro}\}.
\end{equation}
Para cada $a\in B$ definimos el subconjunto básico
\begin{equation}\label{eq:def-a-gorro}
\widehat{a} := \{U\in X_B : a\in U\}\subseteq X_B,
\end{equation}
al que llamaremos la arenilla asociada a $a$.
La familia $\{\widehat{a}\}_{a\in B}$ genera una topología en $X_B$; sus elementos son clopens y constituyen fragmentos elementales de la estructura lógica de la piedra $X_B$.
\end{definition}

\begin{lemma}\label{lem:identidades-hat}
Para todo $a,b\in B$ se tiene
\begin{align}
\widehat{0} &= \varnothing,\label{eq:hat-0}\\
\widehat{1} &= X_B,\label{eq:hat-1}\\
\widehat{a\wedge b} &= \widehat{a}\cap \widehat{b},\label{eq:hat-meet}\\
\widehat{a\vee b} &= \widehat{a}\cup \widehat{b},\label{eq:hat-join}\\
\widehat{\neg a} &= X_B\setminus \widehat{a}.\label{eq:hat-complement}
\end{align}
\end{lemma}

\begin{proof}
Las igualdades \eqref{eq:hat-0} y \eqref{eq:hat-1} son inmediatas. La igualdad \eqref{eq:hat-meet} se deduce de que $U$ es filtro si y sólo si $a,b\in U$ implica $a\wedge b\in U$, y recíprocamente. La igualdad \eqref{eq:hat-join} se sigue del hecho ya usado varias veces: $a\vee b\in U$ si y sólo si $a\in U$ o $b\in U$ (pues en caso contrario $\neg a,\neg b\in U$ y $\neg(a\vee b)=\neg a\wedge \neg b\in U$ produciría $0\in U$). La igualdad \eqref{eq:hat-complement} se obtiene de \eqref{eq:ultrafiltro-decide}.
\end{proof}

En particular, cada $\widehat{a}$ es un conjunto \emph{clopen}. Definimos
\begin{equation}\label{eq:eta-B-def}
\eta_B:B\to \Clop(X_B),\qquad \eta_B(a) := \widehat{a}.
\end{equation}
Llamaremos a este morfismo el \textit{operador arenilla} del álgebra de Boole $B$.

\subsection{Propiedades topológicas}\label{subsec:stone-prop}

\begin{theorem}\label{thm:XB-Stone}
Para toda álgebra de Boole $B$, el espacio $X_B$ es compacto, Hausdorff y totalmente desconectado.
\end{theorem}

\begin{proof}
Consideremos el producto
\[
2^B := \prod_{a\in B} 2_a,
\]
con cada $2_a=\{0,1\}$ discreto, y topología producto. Por Tychonoff, $2^B$ es compacto Hausdorff y totalmente desconectado.

La Proposición~\ref{prop:ultrafiltros-hom-2} identifica $X_B$ con el subconjunto $H\subseteq 2^B$ de homomorfismos $v:B\to 2$. Las condiciones
\[
v(0)=0,\ v(1)=1,\ v(a\wedge b)=v(a)\wedge v(b),\ \dots
\]
definen ecuaciones finitas en coordenadas, de modo que $H$ es intersección de subespacios cerrados de $2^B$, y por tanto es cerrado. Como subespacio cerrado de un compacto Hausdorff totalmente desconectado, $H\simeq X_B$ hereda esas propiedades.

La propiedad de Hausdorff y la total desconexión también se pueden verificar directamente con los clopens $\widehat{a}$.
\end{proof}

\subsection{Teorema de representación}\label{subsec:stone-repr}

\begin{theorem}[Stone, 1936]\label{thm:stone-representation}
El morfismo $\eta_B:B\to \Clop(X_B)$ dado por $a\mapsto \widehat{a}$ es un isomorfismo de álgebras de Boole.
\end{theorem}

\begin{proof}
Ya vimos que $\eta_B$ es homomorfismo. Falta probar inyectividad y sobreyectividad.

\smallskip

\emph{Inyectividad.}  
Supongamos que $\widehat{a}=\widehat{b}$ y $a\not\leq b$. Entonces $c:=a\wedge \neg b\neq 0$. El filtro generado por $c$ es propio, y se puede extender a un ultrafiltro $U$ con $c\in U$. Como $c\leq a$ y $U$ es filtro, se tiene $a\in U$. Además, como $c\leq \neg b$, tenemos $c\wedge b = a\wedge \neg b \wedge b = 0$. Si $b\in U$, entonces $c\wedge b = 0\in U$, contradicción. Por tanto $b\notin U$. Así, $a\in U$ y $b\notin U$, lo que implica $U\in \widehat{a}\setminus \widehat{b}$, contradicción con $\widehat{a}=\widehat{b}$. Por tanto $a\leq b$. Un argumento simétrico da $b\leq a$, así que $a=b$.

\smallskip

\emph{Sobreyectividad.}  
Sea $C\subseteq X_B$ un clopen. La familia $\{\widehat{a}\}_{a\in B}$ es una base de abiertos, así que
\[
C = \bigcup_{a\in S} \widehat{a}
\]
para algún $S\subseteq B$. Como $C$ es compacto (cerrado en $X_B$), existe un subrecubrimiento finito
\[
C=\widehat{a_1}\cup\cdots\cup\widehat{a_n}.
\]
Definimos $a:=a_1\vee\cdots\vee a_n$. Entonces, usando \eqref{eq:hat-join},
\[
\widehat{a}=\widehat{a_1\vee\cdots\vee a_n}
=\widehat{a_1}\cup\cdots\cup\widehat{a_n} = C.
\]
Por tanto, $C$ está en la imagen de $\eta_B$. Esto prueba la sobreyectividad.
\end{proof}

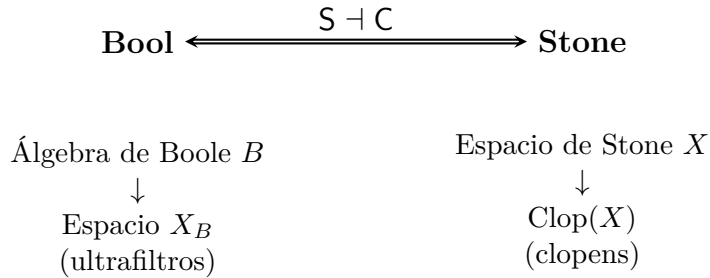
\begin{figure}[h]
  \centering
  \begin{tikzpicture}[>=stealth, node distance=4.5cm, auto]
    \node (B) {$\Bool$};
    \node (Stone) [right=of B] {$\Stone$};
    
    \draw[<->,double,thick] (B) to node[above] {$\mathsf{S} \dashv \mathsf{C}$} (Stone);
    
    \node[below=0.8cm of B] {
      \begin{minipage}{3.5cm}
        \centering
        \small
        Álgebra de Boole $B$\\
        $\downarrow$\\
        Espacio $X_B$\\
        (ultrafiltros)
      \end{minipage}
    };
    
    \node[below=0.8cm of Stone] {
      \begin{minipage}{3.5cm}
        \centering
        \small
        Espacio de Stone $X$\\
        $\downarrow$\\
        $\Clop(X)$\\
        (clopens)
      \end{minipage}
    };
  \end{tikzpicture}
  \caption{La dualidad de Stone establece una anti-equivalencia entre la categoría de álgebras de Boole y la categoría de espacios de Stone. El funtor $\mathsf{S}$ asocia a cada álgebra de Boole $B$ su espacio de Stone $X_B$ (el conjunto de ultrafiltros con topología de clopens), mientras que $\mathsf{C}$ recupera el álgebra de Boole de clopens de un espacio de Stone.}
  \label{fig:stone-duality}
\end{figure}

\begin{figure}[h]
  \centering
  \begin{tikzpicture}[>=stealth, node distance=4.5cm, auto]
    \node (B) {$B$};
    \node (XB) [right=of B] {$X_B$};
    
    \draw[->] (B) -- node[above]{\small $a \mapsto \widehat{a}$} (XB);

    \node[below=0.8cm of B, align=left] (lefttext) {
      $\bullet$ $a \wedge b$ \\[0.05cm]
      $\bullet$ $a \vee b$ \\[0.05cm]
      $\bullet$ $\neg a$
    };
    \node[below=0.8cm of XB, align=left] (righttext) {
      $\bullet$ $\widehat{a}\cap \widehat{b}$ \\[0.05cm]
      $\bullet$ $\widehat{a}\cup \widehat{b}$ \\[0.05cm]
      $\bullet$ $X_B\setminus \widehat{a}$
    };

    \draw[->] (lefttext.east) -- (righttext.west);
  \end{tikzpicture}
  \caption{Diccionario visual de Stone: las operaciones booleanas en $B$ corresponden a operaciones con clopens en $X_B$.}
  \label{fig:stone-dict}
\end{figure}
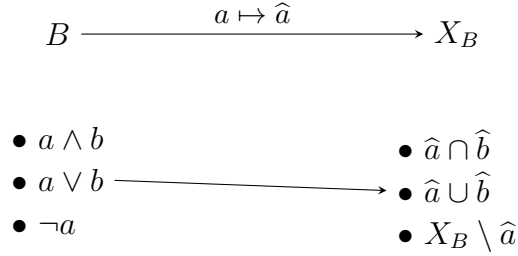

\section{Recuperar un espacio de Stone desde sus clopens}\label{sec:recuperar-espacio}

Sea ahora $X$ un espacio compacto, Hausdorff y totalmente desconectado.

\begin{definition}\label{def:U-x}
Denotamos por $\Clop(X)$ al conjunto de clopens de $X$. Para cada $x\in X$ definimos
\begin{equation}\label{eq:def-Ux}
U_x := \{C\in \Clop(X) : x\in C\}.
\end{equation}
\end{definition}

\begin{proposition}\label{prop:Ux-ultrafiltro}
Para cada $x\in X$, el subconjunto $U_x\subseteq \Clop(X)$ es un ultrafiltro.
\end{proposition}

\begin{proof}
Es un filtro: $X\in U_x$ y $\varnothing\notin U_x$; si $C,D\in U_x$ entonces $x\in C\cap D$ y $C\cap D$ es clopen; si $C\in U_x$ y $C\subseteq D$ con $D$ clopen, entonces $x\in D$ y $D\in U_x$.

Para ver que es ultrafiltro, basta notar que para cada clopen $E$, exactamente uno de $E$ o $E^c$ contiene a $x$, y por tanto pertenece a $U_x$.
\end{proof}

\begin{definition}\label{def:Phi-X}
Definimos
\[
\Phi_X:X\to X_{\Clop(X)},\qquad x\mapsto U_x.
\]
\end{definition}

\begin{lemma}\label{lem:clopens-separan}
En un espacio de Stone $X$, los clopens separan puntos: si $x\neq y$, existe un clopen $C$ tal que $x\in C$ y $y\notin C$.
\end{lemma}

\begin{proof}
$X$ es compacto, Hausdorff y totalmente desconectado. En particular es $0$-dimensional: posee una base de abiertos clopen. Partiendo de cualquier par de abiertos disjuntos que separan a $x,y$, se puede refinar para obtener un clopen que contenga a $x$ y excluya a $y$ (ver \cite{Johnstone1982} para una prueba detallada).
\end{proof}

\begin{theorem}\label{thm:X-isom-XClopX}
Sea $X$ un espacio de Stone. Entonces $\Phi_X:X\to X_{\Clop(X)}$ es un homeomorfismo.
\end{theorem}

\begin{proof}
La continuidad se verifica sobre la base de abiertos de $X_{\Clop(X)}$. Para $C\in \Clop(X)$,
\[
\Phi_X^{-1}(\widehat{C})
= \{x\in X : C\in U_x\}
= \{x\in X : x\in C\}
= C,
\]
que es abierto. Por tanto $\Phi_X$ es continua.

La inyectividad se sigue del Lema~\ref{lem:clopens-separan}: si $x\neq y$, existe un clopen $C$ con $x\in C$, $y\notin C$, por lo que $C\in U_x$ y $C\notin U_y$, de modo que $U_x\neq U_y$.

Para la sobreyectividad, sea $U$ un ultrafiltro de clopens. La familia $\{C\in \Clop(X): C\in U\}$ tiene propiedad de intersección finita. Por compacidad, la intersección $\bigcap_{C\in U} C$ es no vacía. Sea $x$ un punto en esa intersección.

Entonces $C\in U$ implica $x\in C$, así que $U\subseteq U_x$. Por otro lado, si $D$ es clopen con $x\in D$, entonces $D^c$ es clopen y no contiene a $x$, así que $D^c\notin U$. Por ser ultrafiltro, debe cumplirse $D\in U$. Luego $U_x\subseteq U$, y así $U=U_x$.

Finalmente, $X$ es compacto y $X_{\Clop(X)}$ es Hausdorff; una biyección continua entre un compacto y un Hausdorff es un homeomorfismo.
\end{proof}

\section{Espacios profinitos y límites inversos}\label{sec:profinite}

Hasta ahora hemos desarrollado la teoría de espacios de Stone desde una
perspectiva puramente topológica y lógica. En esta sección cambiamos de enfoque:
los espacios profinitos son precisamente los espacios de Stone, pero vistos
desde una perspectiva aritmética: como límites inversos de sistemas finitos.
Esta perspectiva es natural porque muchos objetos aritméticos importantes
(enteros $p$-ádicos, grupos de Galois, completaciones profinitas) se construyen
precisamente como límites inversos de objetos finitos.

\subsection{Sistemas proyectivos y límites inversos}\label{subsec:limites-inversos}

\begin{definition}\label{def:sistema-proyectivo}
Sea $(I,\leq)$ un conjunto dirigido. Un \textit{sistema proyectivo} de espacios topológicos indexado por $I$ consiste en:
\begin{itemize}
  \item una familia de espacios $(X_i)_{i\in I}$;
  \item aplicaciones continuas $\phi_{ji}:X_j\to X_i$ para cada par $i\leq j$,
\end{itemize}
satisfaciendo
\[
\phi_{ii} = \mathrm{id}_{X_i},\qquad
\phi_{ki} = \phi_{ji}\circ \phi_{kj} \quad (i\leq j\leq k).
\]
\end{definition}

\begin{definition}[Límite inverso]\label{def:inverse-limit}
El \textit{límite inverso} $X := \varprojlim_{i\in I} X_i$ es el subespacio de $\prod_{i\in I} X_i$ formado por las sucesiones compatibles
\[
X = \left\{ (x_i)_{i\in I} \in \prod_{i\in I} X_i \,\middle|\, \phi_{ji}(x_j) = x_i\ \forall i\leq j\right\},
\]
con la topología subespacio.
\end{definition}

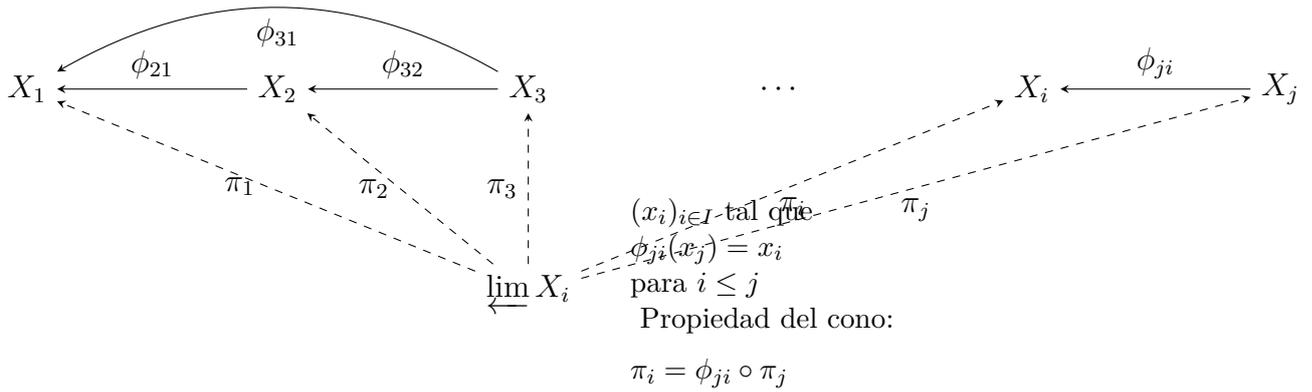
\begin{figure}[h]
  \centering
  \begin{tikzpicture}[>=stealth, node distance=2.5cm, auto]
    \node (X1) {$X_1$};
    \node (X2) [right=of X1] {$X_2$};
    \node (X3) [right=of X2] {$X_3$};
    \node (dots) [right=of X3] {$\cdots$};
    \node (Xi) [right=of dots] {$X_i$};
    \node (Xj) [right=of Xi] {$X_j$};
    \node (dots2) [right=of Xj] {$\cdots$};
    
    \node (Xlim) [below=2cm of X3] {$\varprojlim X_i$};
    
    \draw[->] (X2) to node[above] {$\phi_{21}$} (X1);
    \draw[->] (X3) to node[above] {$\phi_{32}$} (X2);
    \draw[->] (X3) to[bend right=30] node[below] {$\phi_{31}$} (X1);
    \draw[->] (Xj) to node[above] {$\phi_{ji}$} (Xi);
    
    \draw[->,dashed] (Xlim) to node[left] {$\pi_1$} (X1);
    \draw[->,dashed] (Xlim) to node[left] {$\pi_2$} (X2);
    \draw[->,dashed] (Xlim) to node[left] {$\pi_3$} (X3);
    \draw[->,dashed] (Xlim) to node[below] {$\pi_i$} (Xi);
    \draw[->,dashed] (Xlim) to node[below] {$\pi_j$} (Xj);
    
    \node[right=0.5cm of Xlim] {
      \begin{minipage}{4cm}
        \small
        $(x_i)_{i\in I}$ tal que\\
        $\phi_{ji}(x_j) = x_i$\\
        para $i \leq j$\\
        \vspace{0.2cm}
        Propiedad del cono:\\
        $\pi_i = \phi_{ji} \circ \pi_j$
      \end{minipage}
    };
  \end{tikzpicture}
  \caption{Esquema de un sistema proyectivo y su límite inverso. El límite inverso $\varprojlim X_i$ consiste en las sucesiones $(x_i)$ que son compatibles con los morfismos de transición $\phi_{ji}$. Las flechas punteadas $\pi_i$ desde el límite hacia los objetos $X_i$ son las proyecciones canónicas que satisfacen la propiedad del cono: $\pi_i = \phi_{ji} \circ \pi_j$ para $i \leq j$.}
  \label{fig:inverse-limit}
\end{figure}

\subsection{Definición de espacio profinito}\label{subsec:def-profinite}

\begin{definition}\label{def:profinite-space}
Un \textit{espacio profinito} es un espacio homeomorfo al límite inverso de un sistema proyectivo de conjuntos finitos discretos.
La categoría $\ProFin$ tiene como objetos los espacios profinitos y morfismos las aplicaciones continuas.
\end{definition}

\begin{remark}\label{rem:profinite-compacto}
Si cada $X_i$ es finito discreto, entonces es compacto Hausdorff y totalmente desconectado. El producto $\prod_i X_i$ y cualquier subespacio cerrado, en particular el límite inverso, comparten estas propiedades. Por tanto todo profinito es un espacio de Stone.
\end{remark}

\begin{example}[Los enteros $p$-ádicos]\label{ex:Zp-profinito}
Sea $p$ un primo. Consideremos los anillos finitos $\mathbb{Z}/p^n\mathbb{Z}$ con topología discreta y reducción
\[
\phi_{n+1,n}:\mathbb{Z}/p^{n+1}\mathbb{Z}\to \mathbb{Z}/p^n\mathbb{Z},\qquad x\bmod p^{n+1}\mapsto x\bmod p^n.
\]
El límite inverso
\[
\Zp := \varprojlim_{n\geq 1} \mathbb{Z}/p^n\mathbb{Z}
= \left\{(x_n)_{n\geq 1}\in \prod_{n\geq 1}\mathbb{Z}/p^n\mathbb{Z} : x_{n+1}\equiv x_n\pmod{p^n}\right\}
\]
es el anillo de enteros $p$-ádicos con su topología usual, y es profinito.
\end{example}

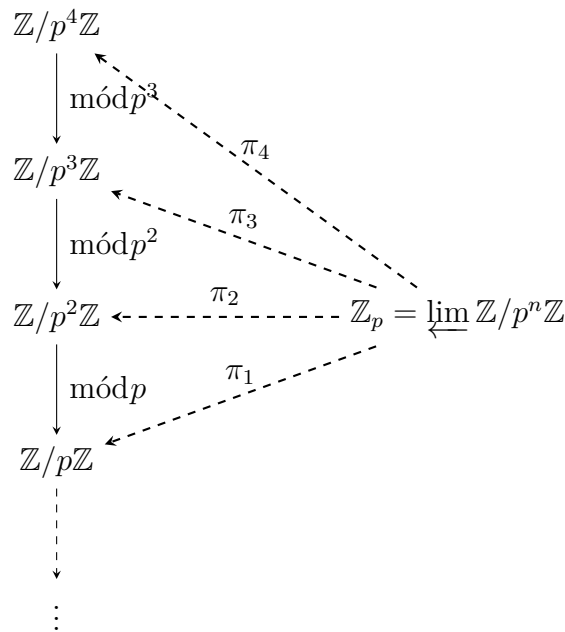
\begin{figure}[h]
  \centering
  \begin{tikzpicture}[>=stealth, node distance=1.2cm]
    \node (Z4) {$\mathbb{Z}/p^4\mathbb{Z}$};
    \node (Z3) [below=of Z4] {$\mathbb{Z}/p^3\mathbb{Z}$};
    \node (Z2) [below=of Z3] {$\mathbb{Z}/p^2\mathbb{Z}$};
    \node (Z1) [below=of Z2] {$\mathbb{Z}/p\mathbb{Z}$};
    \node (dots) [below=of Z1] {$\vdots$};
    \node (Zp) [right=3cm of Z2] {$\Zp = \varprojlim \mathbb{Z}/p^n\mathbb{Z}$};
    
    \draw[->] (Z4) to node[right] {$\bmod p^3$} (Z3);
    \draw[->] (Z3) to node[right] {$\bmod p^2$} (Z2);
    \draw[->] (Z2) to node[right] {$\bmod p$} (Z1);
    \draw[->,dashed] (Z1) to (dots);
    
    \draw[->,thick,dashed] (Zp) to node[above] {$\pi_4$} (Z4);
    \draw[->,thick,dashed] (Zp) to node[above] {$\pi_3$} (Z3);
    \draw[->,thick,dashed] (Zp) to node[above] {$\pi_2$} (Z2);
    \draw[->,thick,dashed] (Zp) to node[above] {$\pi_1$} (Z1);
  \end{tikzpicture}
  \caption{Esquema del sistema proyectivo que define los enteros $p$-ádicos. Cada nivel representa las clases de congruencia módulo $p^n$, y las flechas hacia abajo indican las reducciones naturales $\varphi_{n+1,n}:\mathbb{Z}/p^{n+1}\mathbb{Z}\to\mathbb{Z}/p^n\mathbb{Z}$. Las flechas punteadas gruesas $\pi_n$ desde $\Zp$ hacia cada $\mathbb{Z}/p^n\mathbb{Z}$ son las proyecciones canónicas que satisfacen $\pi_n = \varphi_{n+1,n} \circ \pi_{n+1}$. Un elemento de $\Zp$ corresponde a una sucesión compatible $(x_1, x_2, x_3, \ldots)$ donde $x_n \in \mathbb{Z}/p^n\mathbb{Z}$ y $x_{n+1} \equiv x_n \pmod{p^n}$.}
  \label{fig:Zp-tree}
\end{figure}

\subsection{Equivalencia: espacios de Stone y profinitos}\label{subsec:stone-profinite}

\begin{proposition}\label{prop:profinite-stone}
Un espacio topológico $X$ es profinito si y sólo si es compacto, Hausdorff y totalmente desconectado, es decir, si y sólo si es un espacio de Stone.
\end{proposition}

\begin{proof}
La dirección $(\Rightarrow)$ es la Observación~\ref{rem:profinite-compacto}.

Para $(\Leftarrow)$, sea $X$ compacto, Hausdorff y totalmente desconectado. Consideremos el conjunto $\mathcal{P}$ de todas las particiones finitas de $X$ en clopens y definamos un sistema proyectivo de conjuntos finitos como en la demostración estándar \cite{Johnstone1982}: a cada partición $\mathcal{U}=\{U_1,\dots,U_n\}$ se le asocia $F_{\mathcal{U}}=\{1,\dots,n\}$ y el mapa $p_{\mathcal{U}}:X\to F_{\mathcal{U}}$ que manda $x$ al índice del bloque que lo contiene; refinamientos dan mapas $F_{\mathcal{V}}\to F_{\mathcal{U}}$, y se obtiene un límite inverso $Y$.

El mapa
\[
\Psi:X\to Y,\qquad x\mapsto (p_{\mathcal{U}}(x))_{\mathcal{U}\in\mathcal{P}}
\]
es homeomorfismo; esto usa la separación por clopens, la compacidad y la propiedad universal del límite inverso. Por tanto $X$ es profinito.
\end{proof}

\section{Compactificación de Stone--Čech}\label{sec:stone-cech}

Llegamos ahora a la aparición más contundente de la lógica booleana en topología:
la compactificación de Stone--Čech. Esta construcción ocupa un lugar central en aritmética
porque proporciona una manera de "completar" objetos discretos (como $\mathbb{Z}$ o
$\mathbb{N}$) en espacios compactos que preservan toda la información booleana.
Para nuestras aplicaciones aritméticas bastará trabajar con espacios
\emph{discretos}, que son los más relevantes en teoría de números y geometría
aritmética.

\subsection{Compactificaciones como objetos universales}\label{subsec:stone-cech-universal}

Sea $X$ un espacio topológico. Una \textit{compactificación} de $X$ es un par $(K,\eta)$ donde:
\begin{itemize}
  \item $K$ es un espacio compacto Hausdorff;
  \item $\eta:X\to K$ es continua, inyectiva y con imagen densa.
\end{itemize}

Dadas dos compactificaciones $(K,\eta)$ y $(L,\lambda)$ de $X$, un \textit{morfismo de compactificaciones} $(K,\eta)\to (L,\lambda)$ es una aplicación continua $f:K\to L$ tal que $f\circ \eta = \lambda$.

\begin{definition}\label{def:Cat-CompX}
Sea $X$ un espacio topológico. Definimos la categoría $\mathrm{Comp}(X)$:
\begin{itemize}
  \item objetos: compactificaciones $(K,\eta)$ de $X$;
  \item morfismos: morfismos de compactificaciones.
\end{itemize}
\end{definition}

\begin{definition}\label{def:stone-cech-def}
Una \textit{compactificación de Stone--Čech} de $X$ es un objeto terminal en $\mathrm{Comp}(X)$, es decir, una compactificación $(\beta X,\iota)$ tal que para cualquier otra compactificación $(K,\eta)$ existe un único morfismo $f:(K,\eta)\to (\beta X,\iota)$.
\end{definition}

Equivalente y más operativo:

\begin{proposition}[Propiedad universal]\label{prop:betaX-universal}
Sea $X$ un espacio topológico. Una compactificación $(\beta X,\iota)$ de $X$ es de Stone--Čech si y sólo si para todo compacto Hausdorff $K$ y toda aplicación continua $f:X\to K$ existe una única aplicación continua
\[
\overline{f}:\beta X\to K
\]
tal que $\overline{f}\circ \iota = f$.
\end{proposition}

En términos categóricos, si $\CHaus$ denota la categoría de espacios compactos Hausdorff y $\Top$ la de espacios topológicos, la compactificación de Stone--Čech define un funtor
\[
\beta: \Top_{\mathrm{Tych}}\to \CHaus
\]
(en la subcategoría de espacios completamente regulares Hausdorff) que es \emph{adjunto izquierdo} de la inclusión $\CHaus\hookrightarrow \Top_{\mathrm{Tych}}$. En este artículo trabajaremos sólo con el caso donde $X$ es discreto, y la situación se simplifica enormemente.

\begin{figure}[h]
  \centering
  \begin{tikzcd}[row sep=large, column sep=large]
    S \arrow[r, "i"] \arrow[dr, "f"'] & \beta S \arrow[d, "\tilde{f}"] \\
    & K
  \end{tikzcd}
  \caption{Propiedad universal de la compactificación de Stone--Čech: toda aplicación $f:S\to K$ a un compacto Hausdorff $K$ se extiende de manera única a $\tilde{f}:\beta S\to K$.}
  \label{fig:betaS-universal}
\end{figure}

\subsection{Construcción para espacios discretos}\label{subsec:stone-cech-discreto}

Sea $X$ un conjunto discreto. La potencia $\mathcal{P}(X)$ es una álgebra de Boole, y su espacio de Stone $X_{\mathcal{P}(X)}$ se describe como el conjunto de ultrafiltros sobre $\mathcal{P}(X)$.

\begin{definition}\label{def:betaX-discreto}
Si $X$ es un conjunto discreto, definimos
\[
\beta X := X_{\mathcal{P}(X)} = \{U \subseteq \mathcal{P}(X) : U \textit{ es un ultrafiltro sobre } X\}.
\]
Damos a $\beta X$ la topología generada por los conjuntos básicos
\[
\widehat{A} := \{U\in \beta X : A\in U\},\qquad A\subseteq X.
\]
Definimos la aplicación
\[
\iota:X\to \beta X,\qquad x\mapsto U_x := \{A\subseteq X : x\in A\}.
\]
\end{definition}

\begin{lemma}\label{lem:iota-emb-denso}
Para $X$ discreto, la aplicación $\iota:X\to \beta X$ es inyectiva, continua (trivialmente), inmersión topológica y con imagen densa.
\end{lemma}

\begin{proof}
Si $x\neq y$, el subconjunto $A=\{x\}$ cumple que $A\in U_x$ y $A\notin U_y$, de donde $U_x\neq U_y$; así $\iota$ es inyectiva.

La continuidad es inmediata porque $X$ es discreto: toda función desde un discreto es continua.

Para ver que es inmersión topológica, basta notar que la topología de $X$ coincide con la inducida: el preimagen de $\widehat{A}$ por $\iota$ es
\[
\iota^{-1}(\widehat{A}) = \{x\in X : U_x\in \widehat{A}\}
= \{x\in X : A\in U_x\}
= A,
\]
y los $A\subseteq X$ son precisamente todos los abiertos de $X$.

Para la densidad, sea $W\subseteq \beta X$ un abierto no vacío. Como los $\widehat{A}$ forman una base, existe $A\subseteq X$ con $\widehat{A}\subseteq W$ y $\widehat{A}\neq \varnothing$. Si $A=\varnothing$, entonces $\widehat{A}=\varnothing$, imposible, así que $A\neq \varnothing$. Elige $x\in A$. Entonces $U_x\in \widehat{A}\subseteq W$ y $U_x=\iota(x)$, de modo que $W$ intersecta la imagen de $\iota$.
\end{proof}

El paso clave es demostrar la propiedad universal de Stone--Čech para esta construcción.

\begin{lemma}[Convergencia de ultrafiltros]\label{lem:ultrafiltro-compacto}
Sea $K$ un espacio compacto Hausdorff y $\mathcal{U}$ un ultrafiltro de subconjuntos de $K$. Entonces la intersección
\[
\bigcap_{A\in \mathcal{U}} \overline{A}
\]
es un singleton $\{y\}$. Decimos que $\mathcal{U}$ \emph{converge} a $y$.
\end{lemma}

\begin{proof}
La familia de cerrados $\{\overline{A}:A\in \mathcal{U}\}$ tiene la propiedad de intersección finita: si $\overline{A_1},\dots,\overline{A_n}$ aparecen, entonces
\[
\overline{A_1}\cap\cdots\cap\overline{A_n}\supseteq \overline{A_1\cap\cdots\cap A_n},
\]
y como $\mathcal{U}$ es filtro, $A_1\cap\cdots\cap A_n\in \mathcal{U}$ y no puede ser vacío. Por compacidad de $K$, la intersección total $\bigcap_{A\in \mathcal{U}} \overline{A}$ es no vacía.

Para la unicidad, supongamos que $x\neq y$ pertenecen a dicha intersección. Como $K$ es Hausdorff, existen abiertos disjuntos $U,V$ en $K$ con $x\in U$ y $y\in V$. Como $\mathcal{U}$ es ultrafiltro de subconjuntos de $K$, exactamente uno de $U$ o $K\setminus U$ pertenece a $\mathcal{U}$.

Si $U\in \mathcal{U}$, entonces $x\in \overline{U}$ (pues $x\in \bigcap_{A\in \mathcal{U}} \overline{A}$). Pero $y\notin \overline{U}$, ya que $V$ es un entorno abierto de $y$ disjunto de $U$, y por tanto $y\notin \overline{U}$. Esto contradice que $y\in \bigcap_{A\in \mathcal{U}} \overline{A}$.

Si $K\setminus U\in \mathcal{U}$, entonces $y\in \overline{K\setminus U}$ (pues $y\in \bigcap_{A\in \mathcal{U}} \overline{A}$). Pero $x\notin \overline{K\setminus U}$, ya que $U$ es un entorno abierto de $x$ disjunto de $K\setminus U$, y por tanto $x\notin \overline{K\setminus U}$. Esto contradice que $x\in \bigcap_{A\in \mathcal{U}} \overline{A}$.

En ambos casos llegamos a contradicción. Por tanto, la intersección sólo puede contener un punto.
\end{proof}

\begin{theorem}\label{thm:betaX-universal-discreto}
Sea $X$ un conjunto discreto y $(\beta X,\iota)$ como en la Definición~\ref{def:betaX-discreto}. Entonces $(\beta X,\iota)$ es la compactificación de Stone--Čech de $X$ en el sentido de la Proposición~\ref{prop:betaX-universal}.
\end{theorem}

\begin{proof}
Sea $K$ un compacto Hausdorff y $f:X\to K$ una aplicación continua (cualquier función, pues $X$ es discreto). Para cada ultrafiltro $U\in \beta X$, consideremos la familia
\[
\mathcal{F}_U := \{ f(A) : A\in U\}
\]
de subconjuntos de $K$. Definimos el ultrafiltro de subconjuntos de $K$
\[
f_\ast(U) := \{B\subseteq K : f^{-1}(B)\in U\}.
\]
Es un ultrafiltro: se verifica directamente que es filtro y que para cada $B$, exactamente uno de $B$ o $K\setminus B$ está en $f_\ast(U)$.

Por el Lema~\ref{lem:ultrafiltro-compacto}, $f_\ast(U)$ converge a un único punto $y\in K$, que definimos como
\[
\overline{f}(U) := y.
\]
Por la definición de convergencia, se caracteriza por la propiedad:
\begin{equation}\label{eq:caracterizacion-beta-f}
\forall\ \textit{vecindario abierto }W \textit{ de } y,\quad f^{-1}(W)\in U.
\end{equation}

Demostremos ahora que $\overline{f}:\beta X\to K$ es continua y que extiende a $f$.

\smallskip

\emph{Extensión.}  
Sea $x\in X$ y consideremos $U_x=\iota(x)$. Sea $y:=\overline{f}(U_x)$. Por \eqref{eq:caracterizacion-beta-f}, para todo entorno abierto $W$ de $y$ tenemos $f^{-1}(W)\in U_x$, es decir, $x\in f^{-1}(W)$, lo cual equivale a $f(x)\in W$. Esto significa que todo entorno de $y$ contiene a $f(x)$, luego $y=f(x)$. Así
\[
\overline{f}\circ \iota (x) = \overline{f}(U_x) = f(x).
\]

\smallskip

\emph{Continuidad.}  
Sea $W\subseteq K$ abierto. Consideremos $A:=f^{-1}(W)\subseteq X$. Afirmamos que
\[
\overline{f}^{-1}(W) = \widehat{A}.
\]
En efecto, si $U\in \widehat{A}$, entonces $A\in U$ y por tanto $f^{-1}(W)\in U$, lo que implica que el punto $y=\overline{f}(U)$ satisface que $W$ es vecindario de $y$; en particular $y\in W$. Así $\overline{f}(U)\in W$, es decir, $U\in \overline{f}^{-1}(W)$.

Recíprocamente, si $U\in \overline{f}^{-1}(W)$, entonces $y=\overline{f}(U)\in W$. Por definición de convergencia, $f^{-1}(W)\in U$, o sea, $A\in U$ y $U\in \widehat{A}$. Esto muestra la igualdad
\[
\overline{f}^{-1}(W) = \widehat{f^{-1}(W)},
\]
que es abierta. Por tanto, $\overline{f}$ es continua.

\smallskip

\emph{Unicidad.}  
Supongamos que $g:\beta X\to K$ es otra aplicación continua con $g\circ \iota = f$. La unicidad de la extensión se garantiza por la densidad de $\iota(X)$ en $\beta X$ (Lema~\ref{lem:iota-emb-denso}) y el hecho de que $K$ es Hausdorff: dos aplicaciones continuas que coinciden en un subconjunto denso de un espacio topológico deben coincidir en todo el espacio cuando el codominio es Hausdorff.

Más explícitamente, sea $U\in \beta X$ y escribamos $y:=g(U)$. Para todo entorno abierto $W$ de $y$, $g^{-1}(W)$ es un abierto de $\beta X$ que contiene a $U$. En particular, existe $A\in U$ con $\widehat{A}\subseteq g^{-1}(W)$ (los $\widehat{A}$ forman una base). Entonces, para todo $x\in A$ se tiene
\[
f(x) = g(\iota(x))\in W,
\]
lo que implica $A\subseteq f^{-1}(W)$ y, por tanto, $f^{-1}(W)\in U$. Esto muestra que $y$ satisface \eqref{eq:caracterizacion-beta-f}, por lo que $y=\overline{f}(U)$. Es decir, $g=\overline{f}$.

\smallskip

Con esto hemos mostrado que $(\beta X,\iota)$ verifica la propiedad universal de la Proposición~\ref{prop:betaX-universal}.
\end{proof}

\begin{remark}\label{rem:betaX-extremadamente-desconectado}
Un hecho fundamental que conecta la compactificación de Stone--Čech con la teoría de espacios extremadamente desconectados (Sección~\ref{sec:extremally-disconnected}) es el siguiente: si $X$ es un conjunto discreto, entonces $\beta X$ es un espacio \emph{extremadamente desconectado}. Esto se deduce inmediatamente del hecho de que $\beta X$ es el espacio de Stone de la álgebra de Boole $\mathcal{P}(X)$, que es completa (admite uniones e intersecciones arbitrarias), y del Corolario~\ref{cor:Boole-completa-ED} que establece que los espacios de Stone extremadamente desconectados corresponden exactamente a las álgebras de Boole completas. Esta propiedad será crucial en la Sección~\ref{sec:extremally-disconnected} cuando estudiemos la proyectividad de $\beta X$ en la categoría de compactos Hausdorff.
\end{remark}

\subsection{Ejemplos aritméticos y relación con completaciones}\label{subsec:stone-cech-aritmetica}

\begin{example}[Espacios finitos]\label{ex:betaX-finito}
Si $X$ es finito discreto, todo ultrafiltro es principal (ejercicio clásico), de modo que
\[
\beta X = \{U_x : x\in X\}\cong X.
\]
La compactificación de Stone--Čech no añade puntos nuevos: un compacto finito ya es "completo".
\end{example}

\begin{example}[$\beta\mathbb{N}$]\label{ex:betaN}
Para $X=\mathbb{N}$ discreto, $\beta\mathbb{N}$ es el espacio de Stone de la álgebra de Boole $\mathcal{P}(\mathbb{N})$. Los ultrafiltros \emph{principales} corresponden a puntos $n\in\mathbb{N}$; los ultrafiltros \emph{no principales} son puntos "en el infinito", que codifican maneras coherentes de decidir qué subconjuntos de $\mathbb{N}$ son "grandes".

La topología de $\beta\mathbb{N}$ tiene como base los conjuntos
\[
\widehat{A} = \{U\in \beta\mathbb{N} : A\in U\},\qquad A\subseteq \mathbb{N}.
\]
Combinada con la estructura semigrupal (la operación $+$ se extiende a $\beta\mathbb{N}$), esta compactificación es una herramienta poderosa en combinatoria aditiva y en resultados de recurrencia (vía ultrafiltros idempotentes, etc.).
\end{example}

\begin{example}[$\beta\mathbb{Z}$ y completación profinita]\label{ex:betaZ-profinita}
Consideremos $\mathbb{Z}$ como grupo aditivo discreto. Por un lado, tenemos la compactificación de Stone--Čech $\beta\mathbb{Z}$; por otro, la completación profinita
\[
\widehat{\mathbb{Z}} := \varprojlim_{n\geq 1} \mathbb{Z}/n\mathbb{Z},
\]
que es un compacto Hausdorff y profinito.

La compactificación de Stone--Čech es universal para \emph{todas} las aplicaciones $\mathbb{Z}\to K$ con $K$ compacto Hausdorff; en cambio, $\widehat{\mathbb{Z}}$ es universal en la subcategoría de grupos profinitos \cite{RibesZalesskii2010}.

Existe una aplicación continua y suprayectiva natural
\[
\pi:\beta\mathbb{Z}\longrightarrow \widehat{\mathbb{Z}}
\]
que extiende la inclusión canónica $\mathbb{Z}\hookrightarrow \widehat{\mathbb{Z}}$.

\smallskip

\emph{Construcción de $\pi$.}  
Para cada $n\geq 1$, consideremos la proyección
\[
r_n:\mathbb{Z}\to \mathbb{Z}/n\mathbb{Z}.
\]
Por la universalidad de $\beta\mathbb{Z}$ (Teorema~\ref{thm:betaX-universal-discreto}), $r_n$ se extiende a una aplicación continua única
\[
\overline{r}_n: \beta\mathbb{Z}\to \mathbb{Z}/n\mathbb{Z}.
\]
Estas aplicaciones son compatibles con los mapas de reducción $\mathbb{Z}/mn\mathbb{Z}\to \mathbb{Z}/n\mathbb{Z}$, de modo que definen un único morfismo
\[
\pi:\beta\mathbb{Z}\to \widehat{\mathbb{Z}}\subseteq \prod_{n\geq 1} \mathbb{Z}/n\mathbb{Z},\qquad
\pi(U) := (\overline{r}_n(U))_{n\geq 1}.
\]

Por construcción, para $z\in \mathbb{Z}$ y su ultrafiltro principal $U_z$ se tiene
\[
\pi(U_z) = (r_n(z))_n,
\]
que identifica a $\mathbb{Z}$ con un subgrupo denso de $\widehat{\mathbb{Z}}$. La imagen de $\pi$ es cerrada (compacta) y contiene a esa copia densa de $\mathbb{Z}$, así que coincide con $\widehat{\mathbb{Z}}$. Por tanto, $\pi$ es suprayectiva.
\end{example}

\begin{remark}[Stone--Čech vs completación profinita]\label{rem:beta-vs-hat}
Desde el punto de vista aritmético, aparecen así dos tipos de "compactificación" de un objeto discreto $X$:
\begin{itemize}
  \item $\beta X$: la compactificación \emph{booleana universal}, controlada por la álgebra de todas las partes $\mathcal{P}(X)$ y la lógica de sus subconjuntos.
  \item $\widehat{X}$: la completación \emph{profinita}, controlada por todos los cocientes finitos de $X$ (grupos, anillos, etc.).
\end{itemize}
En nuestro programa conceptual, $\beta X$ se sitúa del lado lógico (álgebras de Boole, ultrafiltros), mientras que $\widehat{X}$ se sitúa del lado aritmético (límite inverso de cocientes finitos). La aplicación $\pi:\beta\mathbb{Z}\to \widehat{\mathbb{Z}}$ refleja precisamente la tensión y el diálogo entre ambos mundos.
\end{remark}
Véase \cite{ComfortNegrepontis1974} para un tratamiento detallado de $\beta\mathbb{N}$, ultrafiltros y su papel en topología y combinatoria.

\section{Dualidad categórica: \texorpdfstring{$\Bool \simeq \Stone^{\mathrm{op}}$}{Bool ≃ Stoneᵒᵖ}}\label{sec:dualidad-categorica}

Los resultados que hemos desarrollado hasta ahora (el teorema de representación
de Stone y la equivalencia entre espacios de Stone y profinitos) pueden
reformularse en el lenguaje de teoría de categorías como una dualidad. Esta
reformulación categórica no es meramente estética: proporciona el marco necesario
para entender cómo estas construcciones se extienden al contexto de las
Matemáticas Condensadas, que es donde adquieren su máxima relevancia aritmética.
La dualidad categórica también muestra cómo la estructura lógica (álgebras de
Boole) y la estructura topológica (espacios profinitos) son dos caras de la misma
moneda, unificando así los aspectos discretos y continuos que aparecen en
aritmética.

\subsection{Funtores de Stone y de clopens}\label{subsec:funtores-Stone-Clop}

Sea $\Bool$ la categoría de álgebras de Boole y $\Stone$ la de espacios de Stone.

Definimos el funtor contravariante
\[
\mathsf{S}:\Bool\to \Stone
\]
por
\[
\mathsf{S}(B) := X_B,\qquad
\mathsf{S}(\varphi) = \varphi^\ast:X_C\to X_B,\quad
\varphi^\ast(U):=\varphi^{-1}(U),
\]
para $\varphi:B\to C$ homomorfismo de álgebras de Boole.

\begin{lemma}\label{lem:phi-star-ultrafiltro}
Para cualquier homomorfismo $\varphi:B\to C$ y ultrafiltro $U$ de $C$, $\varphi^{-1}(U)$ es un ultrafiltro de $B$, y $\varphi^\ast$ es continua.
\end{lemma}

\begin{proof}
Ya se argumentó que $\varphi^{-1}(U)$ es filtro propio y satisface la caracterización de ultrafiltro. La continuidad se verifica porque
\[
(\varphi^\ast)^{-1}(\widehat{a}) = \widehat{\varphi(a)}.
\]
\end{proof}

Definimos el funtor contravariante
\[
\mathsf{C}:\Stone\to \Bool
\]
por
\[
\mathsf{C}(X) := \Clop(X),\qquad
\mathsf{C}(f) := f^{-1},
\]
para $f:X\to Y$ continuo.

\subsection{Equivalencia}\label{subsec:stone-duality}

Los Teoremas~\ref{thm:stone-representation} y \ref{thm:X-isom-XClopX} se reinterpretan como naturalidades.

\begin{proposition}\label{prop:nat-iso-Stone-Bool}
Para cada álgebra de Boole $B$, el morfismo $\eta_B:B\to \Clop(X_B)$ es la componente en $B$ de una transformación natural
\[
\eta:\mathrm{Id}_{\Bool}\Rightarrow \mathsf{C}\circ \mathsf{S}.
\]
Para cada espacio de Stone $X$, el homeomorfismo $\Phi_X:X\to X_{\Clop(X)}$ es la componente en $X$ de una transformación natural
\[
\Phi:\mathrm{Id}_{\Stone}\Rightarrow \mathsf{S}\circ \mathsf{C}.
\]
\end{proposition}

\begin{proof}
Es una verificación directa de que los diagramas correspondientes conmutan; remitimos a \cite{Johnstone1982} para los detalles.
\end{proof}

\begin{theorem}[Dualidad de Stone]\label{thm:stone-duality-categorical}
Los funtores contravariantes $\mathsf{S}$ y $\mathsf{C}$ establecen una anti-equivalencia de categorías
\[
\Bool \;\simeq\; \Stone^{\mathrm{op}},
\]
y, usando la Proposición~\ref{prop:profinite-stone}, también
\[
\Bool \;\simeq\; \ProFin^{\mathrm{op}}.
\]
\end{theorem}

\begin{remark}\label{rem:morfismos-dualidad}
En esta dualidad, los morfismos en $\Bool$ (homomorfismos de álgebras de Boole)
corresponden biyectivamente a morfismos continuos en $\Stone$: un homomorfismo
$\varphi:B\to C$ induce una aplicación continua $\varphi^\ast:X_C\to X_B$ que
envía cada ultrafiltro $U$ de $C$ al ultrafiltro $\varphi^{-1}(U)$ de $B$. La
continuidad es la contraparte natural de la preservación de operaciones
booleanas: los homomorfismos respetan la estructura algebraica, mientras que las
aplicaciones continuas respetan la estructura topológica codificada por los
clopens.
\end{remark}

\begin{remark}[Paralelo con geometría afín]\label{rem:affine-analogy}
Este resultado es el modelo $0$-dimensional de la anti-equivalencia entre esquemas afines y anillos. Aquí, $\Bool$ juega el papel de \emph{álgebra} y $\Stone$ el de \emph{geometría}, con $\Clop$ como análogo de las funciones regulares y $X_B$ como análogo del espectro de un anillo. Desde este punto de vista, la dualidad de Stone es el caso puramente booleano de la anti\-equivalencia entre álgebra y geometría que subyace a la teoría de esquemas, donde $\Spec(R)$ aparece como ejemplo paradigmático de espacio espectral en el sentido de Hochster \cite{Hochster1969, hartshorne1977}.

\end{remark}

\section{El corazón aritmético: \texorpdfstring{$\Zp$}{Z\_p} y grupos de Galois}\label{sec:ejemplos-aritmeticos}

En esta sección presentamos los ejemplos aritméticos fundamentales que
motivan toda la teoría desarrollada en este artículo. Estos ejemplos ---los
enteros $p$-ádicos y los grupos de Galois--- no son meras ilustraciones, sino
los objetos que justifican la necesidad de trabajar con espacios profinitos y
de Stone en aritmética. Veremos cómo la estructura profinita surge
naturalmente en estos contextos y cómo la dualidad de Stone proporciona una
perspectiva unificada.

\subsection{Los enteros \texorpdfstring{$p$}{p}-ádicos como espacio de Stone}\label{subsec:Zp}

Los enteros $p$-ádicos $\Zp$ son uno de los objetos más importantes en teoría
de números y geometría aritmética. Surgen naturalmente cuando queremos estudiar
congruencias módulo potencias arbitrariamente altas de un primo $p$. 

Formalmente, $\Zp$ se define como el límite inverso
\[
\Zp = \varprojlim_{n\geq 1} \mathbb{Z}/p^n\mathbb{Z},
\]
donde las aplicaciones de transición $\mathbb{Z}/p^{n+1}\mathbb{Z}\to
\mathbb{Z}/p^n\mathbb{Z}$ son las reducciones módulo $p^n$. Un elemento de
$\Zp$ es una sucesión $(a_1, a_2, a_3, \ldots)$ donde $a_n\in\mathbb{Z}/p^n\mathbb{Z}$
y $a_{n+1}\equiv a_n\pmod{p^n}$.

Por la Proposición~\ref{prop:profinite-stone}, $\Zp$ es un espacio de Stone:
es compacto, Hausdorff y totalmente desconectado. Su álgebra de Boole dual
$\Clop(\Zp)$ está generada por las \emph{bolas $p$-ádicas}
\[
a+p^n\Zp = \{x\in \Zp : x\equiv a\pmod{p^n}\},
\]
donde $a\in\mathbb{Z}$ y $n\geq 1$, y sus uniones finitas.

\begin{remark}\label{rem:Zp-aritmetica}
La estructura profinita de $\Zp$ refleja su naturaleza aritmética: cada nivel
$n$ corresponde a la información "módulo $p^n$", y la topología captura
precisamente cómo esta información se filtra a través de los niveles. Esta
estructura aparece constantemente en geometría aritmética, donde los enteros
$p$-ádicos aparecen como anillos de valuación en la teoría de esquemas y como
coeficientes en cohomologías $p$-ádicas.

Desde la perspectiva de la dualidad de Stone, el \emph{Lema de Hensel} ---que
garantiza la existencia y unicidad de soluciones de ecuaciones polinomiales en
$\Zp$ bajo ciertas condiciones de reducción módulo $p$--- puede entenderse como
una propiedad que "vive" en la estructura de Stone de $\Zp$: la información
sobre soluciones módulo $p^n$ se propaga coherentemente a través de los niveles
del sistema proyectivo, reflejando la completitud del álgebra de Boole de
clopens.
\end{remark}

\subsection{Grupos de Galois absolutos}\label{subsec:galois}

Sea $K$ un campo (por ejemplo, $\mathbb{Q}$, $\mathbb{Q}_p$, o un campo de
funciones sobre un cuerpo finito). El \emph{grupo de Galois absoluto} de $K$ es
\[
G_K := \Gal(\overline{K}/K),
\]
donde $\overline{K}$ denota la clausura algebraica de $K$. Este grupo controla
todas las extensiones algebraicas de $K$ y ocupa un lugar central en teoría de números y
geometría aritmética.

El grupo $G_K$ tiene estructura profinita
\[
G_K \cong \varprojlim_{L/K \textit{ finita}} \Gal(L/K),
\]
donde $L$ recorre las extensiones finitas de $K$ y el límite inverso se toma
respecto a las restricciones naturales. Esta estructura profinita es esencial:
la topología profinita de $G_K$ es precisamente la que hace que la
correspondencia de Galois se extienda a extensiones infinitas: los subgrupos
cerrados de $G_K$ corresponden biyectivamente a subcampos intermedios entre $K$
y $\overline{K}$.

\begin{remark}\label{rem:galois-aritmetica}
En geometría aritmética, los grupos de Galois absolutos aparecen
constantemente. Por ejemplo, cuando estudiamos una variedad algebraica $X$
sobre un campo $K$, el grupo $G_K$ actúa sobre los puntos geométricos de $X$
(es decir, puntos con coordenadas en $\overline{K}$), y esta acción codifica
información aritmética profunda sobre $X$. La estructura profinita de $G_K$
permite usar técnicas de teoría de grupos profinitos para estudiar estas
acciones (ver \cite{RibesZalesskii2010} para una exposición detallada de la
teoría de grupos profinitos y su aplicación a grupos de Galois).

Un ejemplo canónico particularmente importante es el grupo de Galois absoluto
de un campo finito $\mathbb{F}_q$: se tiene que $G_{\mathbb{F}_q} \cong
\widehat{\mathbb{Z}}$, donde el generador topológico corresponde al
automorfismo de Frobenius $x\mapsto x^q$. Este isomorfismo muestra cómo los
enteros profinitos aparecen naturalmente en aritmética, conectando la
estructura aditiva de $\widehat{\mathbb{Z}}$ con la estructura multiplicativa
de los grupos de Galois.
\end{remark}

\subsection{Aplicaciones en geometría aritmética}\label{subsec:geometria-aritmetica}

Los espacios profinitos y de Stone aparecen de manera natural en varios
contextos de geometría aritmética:

\begin{enumerate}
  \item \emph{Cohomología $p$-ádica}: Cuando estudiamos la cohomología de
  variedades algebraicas, los enteros $p$-ádicos $\Zp$ aparecen como
  coeficientes. La estructura profinita de $\Zp$ permite desarrollar una teoría
  de cohomología $p$-ádica que combina información de todos los niveles
  módulo $p^n$.
  
  \item \emph{Teoría de representaciones de Galois}: Los grupos de Galois
  absolutos $G_K$ actúan sobre espacios vectoriales, y estas representaciones
  codifican información aritmética. La estructura profinita de $G_K$ permite
  estudiar estas representaciones mediante técnicas de grupos profinitos.
  
  \item \emph{Completaciones y límites inversos}: Muchas construcciones en
  geometría aritmética involucran límites inversos de objetos finitos (por
  ejemplo, sistemas de compatibilidad de cohomologías, torres de extensiones de
  Galois). La teoría de espacios profinitos proporciona el marco natural para
  entender estas construcciones.
\end{enumerate}

Estos ejemplos muestran que los espacios profinitos no son meras curiosidades
topológicas, sino objetos que surgen naturalmente en aritmética y geometría
aritmética. La dualidad de Stone proporciona una perspectiva unificada que
conecta la estructura lógica (álgebras de Boole) con la estructura topológica
(espacios profinitos) y la estructura aritmética (límites inversos de objetos
finitos).

\section{Espacios extremadamente desconectados, completitud y proyectividad}\label{sec:extremally-disconnected}

En las secciones anteriores vimos que las álgebras de Boole y los espacios de Stone forman dos caras de una misma moneda: la lógica se geometriza en forma de espacios compactos, Hausdorff y totalmente desconectados. En esta sección refinamos esta dualidad en tres direcciones:

\begin{enumerate}
  \item Caracterizamos una clase más rígida de espacios de Stone, los \emph{espacios extremadamente desconectados} (ED).
  \item Probamos en detalle que estos corresponden exactamente a las \emph{álgebras de Boole completas}.
  \item Explicamos su papel categórico como objetos \emph{proyectivos} en la categoría $\CHaus$ de compactos Hausdorff, y su relación con la compactificación de Stone--Čech.
\end{enumerate}

Desde el punto de vista narrativo, los espacios ED son el "extremo lógico": son aquellos en los que ninguna información topológica queda fuera del alcance de los clopens. Del lado algebraico, esto se traduce en que la álgebra booleana subyacente no tiene "lagunas": admite supremos e ínfimos de \emph{todas} las familias.

\subsection{Abiertos regulares y completitud booleana}\label{subsec:reg-open}

Sea $X$ un espacio topológico. Para cualquier subconjunto $U\subseteq X$ definimos el operador de \emph{regularización}:
\begin{equation}\label{eq:regularizacion}
  r(U) := \operatorname{int}(\overline{U}),
\end{equation}
donde $\operatorname{int}$ denota el interior y $\overline{U}$ la clausura.

\begin{definition}[Abiertos regulares]\label{def:regular-open}
Un abierto $U\subseteq X$ se llama \textit{regular abierto} si satisface
\[
U = \operatorname{int}\big(\overline{U}\big),
\quad \textit{es decir, } U = r(U).
\]
Denotamos por $\mathrm{RO}(X)$ al conjunto de todos los abiertos regulares de $X$.
\end{definition}

\begin{lemma}\label{lem:RO-BA}
Para todo espacio topológico $X$, el conjunto $\mathrm{RO}(X)$ es una álgebra de Boole con las operaciones
\begin{align}
  U \wedge V &:= U \cap V, \label{eq:RO-meet}\\
  U \vee V &:= r(U \cup V) = \operatorname{int}\big(\overline{U\cup V}\big), \label{eq:RO-join}\\
  U^{*} &:= \operatorname{int}\big(X \setminus \overline{U}\big),\label{eq:RO-comp}
\end{align}
y neutros $0=\emptyset$, $1=X$.
\end{lemma}

\begin{proof}
Comprobemos primero que $\mathrm{RO}(X)$ es estable bajo estas operaciones.

\smallskip

\emph{Cierre bajo intersección.} Sean $U,V\in \mathrm{RO}(X)$. Entonces $U$ y $V$ son abiertos, luego $U\cap V$ es abierto. Además,
\[
\overline{U\cap V} \subseteq \overline{U}\cap \overline{V},
\]
por lo que
\[
\operatorname{int}\big(\overline{U\cap V}\big) 
\subseteq \operatorname{int}\big(\overline{U}\cap \overline{V}\big)
\subseteq \operatorname{int}(\overline{U}) \cap \operatorname{int}(\overline{V})
= U \cap V.
\]
Por otro lado, siempre se tiene $U\cap V \subseteq \overline{U\cap V}$, así que $U\cap V \subseteq \operatorname{int}(\overline{U\cap V})$. Concluimos
\[
U\cap V = \operatorname{int}\big(\overline{U\cap V}\big),
\]
es decir, $U\cap V$ es regular abierto.

\smallskip

\emph{Cierre bajo $U\vee V$.} Si $U,V$ son regulares, entonces $U\cup V$ es abierto y por definición $U\vee V$ es el interior de su clausura, de modo que $U\vee V\in\mathrm{RO}(X)$.

\smallskip

\emph{Cierre bajo complemento $^{*}$.} Si $U\in \mathrm{RO}(X)$, entonces $X\setminus \overline{U}$ es abierto (complemento de cerrado), y $U^{*}$ es el interior de ese abierto, luego $U^{*}$ es regular abierto.

\smallskip

Las propiedades de álgebra de Boole (asociatividad, conmutatividad, leyes de absorción y distributividad) se verifican a partir de las propiedades estándar de interior, clausura y las operaciones de conjuntos. Por ejemplo, la ley de complementación $U\wedge U^{*} = 0$ proviene de que $U\subseteq \overline{U}$ y $U^{*}\subseteq X\setminus \overline{U}$; la ley $U\vee U^{*} = 1$ se deduce de que
\[
\overline{U\cup U^{*}} = X
\]
y, por tanto, $\operatorname{int}(\overline{U\cup U^{*}})=X$. Dejamos los detalles al lector, pues en su mayoría son verificaciones elementales.
\end{proof}

El siguiente resultado es el punto clave para la completitud.

\begin{proposition}\label{prop:RO-completa}
Para todo espacio topológico $X$, la álgebra de Boole $\mathrm{RO}(X)$ es \textit{completa}: todo subconjunto $\mathcal{U}\subseteq \mathrm{RO}(X)$ admite supremo e ínfimo en $\mathrm{RO}(X)$.
Más precisamente:
\begin{align}
  \bigwedge_{U\in \mathcal{U}} U &= \bigcap_{U\in \mathcal{U}} U,\label{eq:RO-inf}\\
  \bigvee_{U\in \mathcal{U}} U &= r\!\Big(\bigcup_{U\in \mathcal{U}} U\Big) 
  = \operatorname{int}\Big(\overline{\bigcup_{U\in \mathcal{U}} U}\Big). \label{eq:RO-sup}
\end{align}
\end{proposition}

\begin{proof}
Sea $\mathcal{U}\subseteq \mathrm{RO}(X)$.

\smallskip

\emph{Ínfimo.} El conjunto $V := \bigcap_{U\in\mathcal{U}} U$ es intersección arbitraria de abiertos, luego es abierto. Mostremos que $V$ es regular.

Por un lado, siempre se tiene $V\subseteq \operatorname{int}(\overline{V})$. Por otro, como
\[
\overline{V} = \overline{\bigcap_{U\in\mathcal{U}} U}
\subseteq \bigcap_{U\in\mathcal{U}} \overline{U},
\]
obtenemos
\[
\operatorname{int}(\overline{V}) 
\subseteq \operatorname{int}\Big(\bigcap_{U\in\mathcal{U}} \overline{U}\Big)
\subseteq \bigcap_{U\in\mathcal{U}} \operatorname{int}(\overline{U})
= \bigcap_{U\in\mathcal{U}} U = V.
\]
Por tanto $\operatorname{int}(\overline{V}) = V$, así que $V$ es regular abierto y pertenece a $\mathrm{RO}(X)$.

Si $W\in\mathrm{RO}(X)$ está contenido en todos los $U\in\mathcal{U}$, entonces $W\subseteq V$. De aquí se deduce que $V$ es el ínfimo de $\mathcal{U}$.

\smallskip

\emph{Supremo.} Sea $U_0 := \bigcup_{U\in\mathcal{U}} U$, que es abierto, y definamos $W:= r(U_0) = \operatorname{int}(\overline{U_0})$. Entonces $W\in\mathrm{RO}(X)$. Claramente $U\subseteq U_0 \subseteq \overline{U_0}$ para todo $U\in \mathcal{U}$, de modo que $U\subseteq W$ (pues $W$ es un abierto que contiene a $U_0$). Por tanto $W$ es una cota superior de $\mathcal{U}$.

Si $Z\in \mathrm{RO}(X)$ es otra cota superior de $\mathcal{U}$, entonces $U\subseteq Z$ para todo $U\in\mathcal{U}$, y en consecuencia $U_0\subseteq Z$. De aquí se deduce $\overline{U_0}\subseteq \overline{Z}$. Pero $Z$ es regular, luego $Z=\operatorname{int}(\overline{Z})$, y se tiene
\[
W = \operatorname{int}(\overline{U_0}) \subseteq \operatorname{int}(\overline{Z}) = Z.
\]
Por tanto, $W$ es el supremo de $\mathcal{U}$ en $\mathrm{RO}(X)$.
\end{proof}

\begin{remark}\label{rem:Clop-subalgebra-RO}
Si $X$ es Hausdorff y totalmente desconectado (en particular, si es un espacio de Stone), entonces $\Clop(X)$ es un subálgebra de Boole de $\mathrm{RO}(X)$: todo clopen es regular abierto, porque para $C$ clopen se tiene $\overline{C}=C$ y $\operatorname{int}(\overline{C})=C$.
En este contexto, $\mathrm{RO}(X)$ puede interpretarse como una \emph{completación} booleana de la álgebra de clopens $\Clop(X)$.
\end{remark}

\subsection{Espacios ED y completitud de los clopens}\label{subsec:ED-Clop-completa}

Pasamos ahora a la noción central.

\begin{definition}[Espacio extremadamente desconectado]\label{def:ED}
Un espacio topológico $X$ se dice \textit{extremadamente desconectado} (ED) si para todo abierto $U\subseteq X$, su clausura $\overline{U}$ es también un abierto (y en particular, clopen).
\end{definition}

Todo espacio ED es totalmente desconectado, pero la recíproca no vale: el conjunto de Cantor, por ejemplo, es un espacio de Stone que \emph{no} es ED.

En el caso de espacios de Stone, la relación con la completitud de la álgebra de clopens se expresa de la forma siguiente.

\begin{theorem}\label{thm:ED-Clop-completa}
Sea $X$ un espacio de Stone. Entonces son equivalentes:
\begin{enumerate}[label=\roman*)]
  \item $X$ es extremadamente desconectado.
  \item La álgebra de Boole $\Clop(X)$ es completa.
\end{enumerate}
\end{theorem}

\begin{proof}
Recordemos que $X$ es compacto, Hausdorff y totalmente desconectado, y que los clopens forman una base de la topología.

\smallskip

\emph{(i) $\Rightarrow$ (ii).} Supongamos que $X$ es ED.
Por la Observación~\ref{rem:Clop-subalgebra-RO}, sabemos que $\Clop(X)$ es un subálgebra de $\mathrm{RO}(X)$. Por la Proposición~\ref{prop:RO-completa}, $\mathrm{RO}(X)$ es una álgebra de Boole completa. Mostraremos que, bajo la hipótesis ED, toda abierto regular es de hecho clopen; esto implica $\mathrm{RO}(X)=\Clop(X)$ y por tanto completitud de $\Clop(X)$.

Sea $U\in \mathrm{RO}(X)$ un abierto regular. Entonces $U$ es abierto, y por ser $X$ ED, su clausura $\overline{U}$ es también abierta. Tenemos
\[
U = \operatorname{int}(\overline{U})
\]
por regularidad; pero si $\overline{U}$ es abierta, entonces $\operatorname{int}(\overline{U}) = \overline{U}$. Por tanto,
\[
U = \overline{U},
\]
y $U$ es simultáneamente abierto y cerrado, es decir, clopen. Así, $\mathrm{RO}(X)\subseteq \Clop(X)$, y por la Observación~\ref{rem:Clop-subalgebra-RO} también $\Clop(X)\subseteq \mathrm{RO}(X)$, de modo que $\mathrm{RO}(X)=\Clop(X)$.

Como $\mathrm{RO}(X)$ es completa, $\Clop(X)$ también lo es.

\smallskip

\emph{(ii) $\Rightarrow$ (i).}
Supongamos ahora que $\Clop(X)$ es una álgebra de Boole completa.
Sea $U\subseteq X$ un abierto cualquiera y consideremos el conjunto
\[
  \mathcal{C} := \{\, C\in\Clop(X) : C\subseteq U \,\}.
\]
Por completitud, el supremo
\[
  C_0 := \bigvee_{C\in\mathcal{C}} C
\]
existe en $\Clop(X)$ y es, por tanto, un clopen de $X$.

En primer lugar, $U\subseteq C_0$. En efecto, si $x\in U$, como los clopens forman una base de la topología, existe $C_x\in\Clop(X)$ con
\[
x\in C_x\subseteq U,
\]
así que $C_x\in\mathcal{C}$ y, por definición del supremo, $C_x\subseteq C_0$. Por tanto $x\in C_0$.

En segundo lugar, $C_0\subseteq \overline{U}$. Si $x\notin \overline{U}$, existe un clopen $D$ tal que
\[
x\in D\subseteq X\setminus\overline{U}.
\]
Entonces $D$ es disjunto de $U$, y en particular de cada $C\in\mathcal{C}$. Esto implica $D\subseteq X\setminus C$ para todo $C\in\mathcal{C}$, de modo que
\[
D\subseteq \bigwedge_{C\in\mathcal{C}} (X\setminus C).
\]
Por las leyes de De Morgan en álgebras de Boole completas, tenemos $\bigwedge_{C\in\mathcal{C}} (X\setminus C) = X\setminus \bigvee_{C\in\mathcal{C}} C = X\setminus C_0$, de donde
\[
D\subseteq X\setminus C_0.
\]
En particular $x\notin C_0$. Hemos mostrado que
\[
X\setminus \overline{U} \subseteq X\setminus C_0,
\]
es decir, $C_0\subseteq \overline{U}$.

Resumiendo, $U\subseteq C_0\subseteq \overline{U}$ y $C_0$ es clopen. Como $\overline{U}$ es la menor cerrado que contiene a $U$, y $C_0$ es cerrado (al ser clopen) conteniendo a $U$, se tiene $\overline{U}\subseteq C_0$. Por tanto
\[
\overline{U} = C_0,
\]
que es abierto. Esto demuestra que la clausura de todo abierto es abierta, es decir, que $X$ es extremadamente desconectado.
\end{proof}

\begin{definition}[Operador gránulo]\label{def:operador-gránulo}
Sea $X$ un espacio de Stone y sea $B := \Clop(X)$ su álgebra de Boole de clopens. Dada una subálgebra booleana finita $B_0 \subseteq B$, denotemos por
\[
\mathrm{At}(B_0) = \{A_1,\dots,A_n\}
\]
su conjunto finito de átomos, que interpretamos (siguiendo la Observación~\ref{rem:granulos}) como los \emph{gránulos asociados} a $B_0$. Definimos el \textit{operador gránulo} asociado a $B_0$ como la aplicación
\[
\mathsf{Ar}_{B_0} : \mathcal{P}(X) \longrightarrow B,\qquad
\mathsf{Ar}_{B_0}(E) := \bigcup_{\substack{1\leq i\leq n \\ A_i \subseteq \operatorname{int}(\overline{E})}} A_i.
\]
\end{definition}

\begin{remark}\label{rem:prop-gránulo}
El operador $\mathsf{Ar}_{B_0}$ tiene tres propiedades elementales pero conceptualmente útiles:
\begin{enumerate}[label=\roman*)]
\item Es \emph{monótono}: si $E\subseteq F\subseteq X$, entonces $\mathsf{Ar}_{B_0}(E)\subseteq \mathsf{Ar}_{B_0}(F)$.
\item Si $U\in B_0$ es clopen, entonces $\mathsf{Ar}_{B_0}(U)$ es la mayor unión de átomos de $B_0$ contenida en $U$, y coincide con $U$ siempre que $U$ sea unión de átomos de $B_0$ (lo cual ocurre, por ejemplo, si $U$ es uno de los propios átomos).
\item Para un abierto regular $U\subseteq X$, el conjunto $\mathsf{Ar}_{B_0}(U)$ es una aproximación finita, booleana, del abierto $U$ vista con la resolución determinada por la subálgebra $B_0$: se obtiene tomando todos los gránulos que quepan por completo dentro de la "envoltura regular" $r(U)=\operatorname{int}(\overline{U})$.
\end{enumerate}
Geométricamente, puede pensarse en $\mathsf{Ar}_{B_0}$ como un operador que deja caer un puñado finito de \emph{gránulos} sobre el espacio $X$ y registra qué granos quedan completamente contenidos en la región que queremos aproximar. Al refinar $B_0$ (es decir, al considerar subálgebras finitas cada vez más grandes), los gránulos se hacen más pequeños y la aproximación de los abiertos regulares por uniones de átomos se vuelve más precisa, reflejando de manera finita la completitud booleana del espacio extremadamente desconectado.
\end{remark}

Del lado algebraico, esto se traduce así:

\begin{corollary}\label{cor:Boole-completa-ED}
Sea $B$ un álgebra de Boole y $X_B$ su espacio de Stone. Entonces:
\[
B \textit{ es completa } \iff X_B \textit{ es extremadamente desconectado}.
\]
\end{corollary}

\begin{proof}
Por el Teorema de Representación de Stone, tenemos un isomorfismo booleano $B \cong \Clop(X_B)$. El resultado se deduce aplicando el Teorema~\ref{thm:ED-Clop-completa} a $X=X_B$.
\end{proof}

\begin{remark}[Subcategorías duales]\label{rem:CBool-EDStone}
Denotemos por $\mathbf{CBool}$ la categoría de álgebras de Boole completas y homomorfismos que preservan supremos e ínfimos arbitrarios, y por $\mathbf{EDStone}$ la subcategoría plena de $\Stone$ formada por los espacios de Stone extremadamente desconectados. El Corolario~\ref{cor:Boole-completa-ED} muestra que la dualidad de Stone se restringe a una anti-equivalencia
\[
\mathbf{CBool} \;\simeq\; \mathbf{EDStone}^{\mathrm{op}}.
\]
Esta es la versión "completa" de la dualidad de Stone: la completitud lógica (supremos infinitos) en el lado algebraico se refleja en la rigidez extrema de la topología en el lado geométrico.
\end{remark}

\subsection{Compactificación de Stone--Čech y representabilidad}\label{subsec:betaS-ED-repr}

Consideremos ahora un conjunto $S$ con la topología discreta. Su álgebra de Boole natural es la potencia $\mathcal{P}(S)$, que es una álgebra de Boole completa: uniones e intersecciones arbitrarias de subconjuntos siguen siendo subconjuntos.

Por la dualidad de Stone, el espacio de Stone de $\mathcal{P}(S)$ es un espacio compacto Hausdorff totalmente desconectado, que identificamos con la compactificación de Stone--Čech $\beta S$. El Corolario~\ref{cor:Boole-completa-ED} implica inmediatamente:

\begin{proposition}\label{prop:betaS-ED}
Sea $S$ un conjunto discreto. Entonces su compactificación de Stone--Čech $\beta S$ es un espacio extremadamente desconectado.
\end{proposition}

En particular, $\beta\mathbb{N}$ es al mismo tiempo:
\begin{itemize}
  \item un espacio de Stone (compacto, Hausdorff, totalmente desconectado),
  \item extremadamente desconectado (por completitud de $\mathcal{P}(\mathbb{N})$),
  \item un objeto proyectivo en $\CHaus$ (como veremos enseguida).
\end{itemize}

Desde el punto de vista categórico, la compactificación de Stone--Čech se caracteriza por una propiedad universal muy fuerte.

\begin{theorem}[Propiedad universal de $\beta S$]\label{thm:betaS-universal}
Sea $S$ un conjunto discreto y sea $i:S\to \beta S$ la inmersión canónica (identificación de $S$ con un subconjunto denso de $\beta S$). Para todo compacto Hausdorff $K$ y toda aplicación $f:S\to K$, existe una única aplicación continua $\tilde{f}:\beta S \to K$ tal que
\[
\tilde{f}\circ i = f.
\]
Equivalente y categóricamente, el functor que envía un compacto $K$ al conjunto de aplicaciones continuas $\Hom_{\CHaus}(\beta S,K)$ es representable por $S$ vía
\[
\Hom_{\CHaus}(\beta S, K)\;\cong\; \Hom_{\Set}(S, U(K)),
\]
natural en $K$, donde $U:\CHaus\to \Set$ es el funtor olvidadizo.
\end{theorem}

\begin{remark}[Adjunción]\label{rem:betaS-adjuncion}
El teorema anterior puede reformularse diciendo que la compactificación de Stone--Čech define un funtor
\[
\beta: \Set_{\mathrm{disc}} \longrightarrow \CHaus
\]
que es \emph{adjunto izquierdo} al funtor olvidadizo $U:\CHaus\to \Set$. En otras palabras, $\beta S$ es el "compacto Hausdorff libre generado por $S$", y la condición de ser ED refleja el hecho de que ninguna compatibilidad adicional se impone más allá de la compacidad y la Hausdorffidad.
\end{remark}

Desde el punto de vista aritmético, el caso $S=\mathbb{N}$ es especialmente significativo: muchos fenómenos combinatorios y aritméticos (como teoremas de Ramsey, teorema de Hindman, etc.) pueden reformularse como afirmaciones sobre la existencia de puntos en $\beta\mathbb{N}$ con propiedades dinámicas particulares. La rigidez ED de $\beta\mathbb{N}$ garantiza que esas propiedades se expresan en términos de la álgebra booleana completa $\mathcal{P}(\mathbb{N})$.

\subsection{Proyectividad en \texorpdfstring{$\CHaus$}{CHaus} y el Teorema de Gleason}\label{subsec:Gleason}

Finalmente, describimos el lugar de los espacios ED dentro de la categoría $\CHaus$.

\begin{definition}[Objeto proyectivo en $\CHaus$]\label{def:proyectivo-CHaus}
Un espacio compacto Hausdorff $P$ se dice \textit{proyectivo} en $\CHaus$ si para todo diagrama
\[
\begin{tikzcd}[row sep=large, column sep=large]
    & P \arrow[d, "g"] \arrow[dl, dashed, "h"'] \\
    X \arrow[r, "f", two heads] & Y
\end{tikzcd}
\]
donde $f:X\to Y$ es continua y sobreyectiva, y $g:P\to Y$ es continua, existe una aplicación continua $h:P\to X$ tal que $f\circ h = g$.
\end{definition}

Es decir, cualquier mapa desde $P$ hacia un cociente $Y$ puede \emph{levantarse} a través de la sobreyección $X\to Y$. Estos espacios juegan, en $\CHaus$, el mismo rol que los módulos proyectivos en álgebra.

El siguiente resultado de Gleason~\cite{Gleason1958} cierra el círculo.

\begin{theorem}[Gleason]\label{thm:Gleason}
Sea $P$ un espacio compacto Hausdorff. Entonces las condiciones siguientes son equivalentes:
\begin{enumerate}[label=\roman*)]
  \item $P$ es extremadamente desconectado.
  \item $P$ es proyectivo en la categoría $\CHaus$.
\end{enumerate}
\end{theorem}

\begin{remark}[Cubierta de Gleason y relación con Stone--Čech]\label{rem:cubierta-Gleason}
Como consecuencia, para todo compacto Hausdorff $K$ existe un epimorfismo continuo
\[
\pi: P(K) \twoheadrightarrow K
\]
donde $P(K)$ es ED (y por tanto proyectivo), y $\pi$ es esencialmente única con esta propiedad. El espacio $P(K)$ se llama la \emph{cubierta de Gleason} de $K$.

En el caso extremo $K=S$ discreto, la construcción recupera precisamente $\beta S$: la compactificación de Stone--Čech es la cubierta proyectiva "libre" del conjunto discreto $S$, y su estructura ED refleja la completitud de la álgebra de Boole $\mathcal{P}(S)$.
\end{remark}

\begin{exercise}\label{ex:ED-closures-disjoint}
Sea $X$ un espacio ED. Demuestra que si $U,V\subseteq X$ son abiertos disjuntos, entonces sus clausuras $\overline{U}$ y $\overline{V}$ son también disjuntas. (Pista: razona por contradicción usando que las clausuras son clopens.)
\end{exercise}

\begin{exercise}\label{ex:betaN-sucesiones}
Sea $S=\mathbb{N}$ con la topología discreta y considera $\beta\mathbb{N}$. Investiga el comportamiento de las sucesiones convergentes en $\beta\mathbb{N}$. En particular:
\begin{enumerate}[label=\alph*)]
  \item Muestra que la sucesión $(n)_{n\in\mathbb{N}}$ de puntos naturales no converge en $\beta\mathbb{N}$.
  \item Argumenta que cualquier sucesión convergente en $\beta\mathbb{N}$ debe, en cierto sentido fuerte, estabilizarse en torno a un ultrafiltro; discute cómo entra la condición ED en este fenómeno.
\end{enumerate}
\end{exercise}

\section{Epílogo: profinitos como generadores en matemática condensada}\label{sec:condensed}

Hasta aquí hemos desarrollado la teoría clásica de espacios de Stone,
profinitos y la compactificación de Stone--Čech, mostrando cómo estos
conceptos surgen naturalmente en contextos aritméticos. En esta sección
final, mostramos cómo estas construcciones encuentran su lugar natural en el
marco conceptual de las \emph{Matemáticas Condensadas} de Clausen--Scholze,
donde los profinitos no son meros ejemplos, sino los \emph{generadores
fundamentales} de una nueva categoría que unifica topología y álgebra de
manera adecuada para la aritmética contemporánea.

El problema fundamental que motiva este enfoque ya fue señalado en la
Introducción~\ref{sec:intro}: cuando trabajamos con objetos aritméticos
dotados de topologías naturales, la categoría de espacios topológicos
clásicos presenta dificultades técnicas insuperables. La solución de
Clausen--Scholze consiste en reemplazar la categoría de espacios
topológicos por una categoría de \emph{conjuntos condensados}, donde la
información topológica se codifica mediante \emph{sondas profinitas}. En
lugar de estudiar un espacio $X$ directamente, lo examinamos a través de
todas las aplicaciones continuas $S\to X$ donde $S$ recorre los espacios
profinitos.

Esta perspectiva hace que los profinitos, que ya hemos visto aparecer
naturalmente en aritmética (como $\Zp$ en el Ejemplo~\ref{ex:Zp-profinito},
grupos de Galois en la Sección~\ref{subsec:galois}, y completaciones
profinitas en general), jueguen el papel de objetos generadores sobre los
cuales se construye toda la teoría. La dualidad de Stone que hemos
desarrollado (Teorema~\ref{thm:stone-duality-categorical}) se extiende
naturalmente a este marco, estableciendo un puente entre la lógica booleana
clásica y la estructura categórica moderna. Para una exposición más detallada de estos temas, véase \cite{ScholzeCondensed2,ClausenScholzeLecturesAnalytic}.

\subsection{El sitio profinito y su topología de Grothendieck}\label{subsec:sitio-profinito}

Denotamos por $\ProFin$ a la categoría de espacios topológicos profinitos
(compactos, totalmente desconectados y Hausdorff) y aplicaciones continuas
entre ellos. Por la Proposición~\ref{prop:profinite-stone}, estos son
precisamente los espacios de Stone, y por tanto están en dualidad con las
álgebras de Boole (Teorema~\ref{thm:stone-duality-categorical}). Solo los
espacios extremadamente desconectados están en dualidad con álgebras de Boole
completas (Corolario~\ref{cor:Boole-completa-ED}).

Para convertir $\ProFin$ en un \emph{sitio de Grothendieck}, necesitamos
especificar una \emph{topología de Grothendieck} $J$ sobre esta categoría, es
decir, una colección de familias de morfismos (llamadas \emph{cubrimientos})
que satisfacen ciertos axiomas. En el caso profinito, la topología natural es
la siguiente:

\begin{definition}[Topología profinita]\label{def:topologia-profinita}
Un \emph{cubrimiento} de un espacio profinito $S$ es una familia \emph{finita} de
aplicaciones continuas $\{f_i:S_i\to S\}_{i\in I}$ (donde $I$ es un conjunto
finito) tal que:
\begin{enumerate}[label=\roman*)]
  \item cada $f_i$ es un epimorfismo en la categoría de espacios topológicos
  (es decir, $f_i(S_i)=S$);
  \item la imagen conjunta $\bigcup_{i\in I} f_i(S_i)$ es igual a $S$.
\end{enumerate}
La \emph{topología profinita} $J$ es la topología de Grothendieck generada por
estos cubrimientos finitos. La finitud del índice $I$ es esencial: garantiza
que el sitio sea "compacto" en sentido categórico, lo cual facilita
la teoría de gavillas sobre este sitio.
\end{definition}

\begin{remark}\label{rem:topologia-profinita-motivacion}
Esta definición captura la idea de que un espacio profinito puede ser
recubierto por otros profinitos de manera finita. Dado que los profinitos
son límites inversos de espacios finitos, es natural considerar cubrimientos
que reflejen esta estructura: cada cubrimiento finito $\{S_i\to S\}$ puede
pensarse como una manera de aproximar $S$ mediante los $S_i$, y la
condición de epimorfismo garantiza que no se pierde información esencial. Esta construcción es análoga a la topología de Grothendieck en el sitio de esquemas, adaptada al contexto profinito \cite{Mannion1994}.
\end{remark}

El par $(\ProFin,J)$ se llama el \emph{sitio profinito}. Este sitio tiene
propiedades técnicas importantes que lo hacen adecuado para la aritmética:

\begin{proposition}\label{prop:propiedades-sitio-profinita}
El sitio profinito $(\ProFin,J)$ satisface:
\begin{enumerate}[label=\roman*)]
  \item \emph{Finitud}: todo cubrimiento es finito.
  \item \emph{Estabilidad bajo productos fibrados}: si $\{S_i\to S\}$ es un
  cubrimiento y $T\to S$ es cualquier morfismo, entonces $\{S_i\times_S T\to
  T\}$ es un cubrimiento.
  \item \emph{Composición}: si $\{S_i\to S\}$ y $\{S_{ij}\to S_i\}$ son
  cubrimientos, entonces $\{S_{ij}\to S\}$ es un cubrimiento.
\end{enumerate}
\end{proposition}

\begin{proof}
La finitud es por definición. La estabilidad bajo productos fibrados se sigue
de que los productos fibrados de profinitos son profinitos (pues son límites
inversos de productos fibrados finitos) y de que los epimorfismos se preservan
bajo productos fibrados. La composición es una verificación directa.
\end{proof}

\begin{definition}[Conjuntos condensados]\label{def:conjunto-condensado}
Un \emph{conjunto condensado} es una \emph{gavilla de conjuntos} sobre el sitio
profinito $(\ProFin,J)$, es decir, un funtor
\[
X:\ProFin^{\mathrm{op}}\longrightarrow\Set
\]
que satisface la \emph{condición de gavilla}: para todo cubrimiento finito
$\{f_i:S_i\to S\}_{i\in I}$ en $(\ProFin,J)$, el diagrama
\[
X(S) \longrightarrow \prod_{i\in I} X(S_i) \rightrightarrows \prod_{i,j\in I} X(S_i\times_S S_j)
\]
es un ecualizador en $\Set$. Denotamos por $\Cond$ a la categoría de conjuntos
condensados (con morfismos las transformaciones naturales entre funtores).

La categoría $\Cond$ es un \emph{topos de Grothendieck}, y en particular es
\emph{abeliana} cuando consideramos grupos abelianos condensados (objetos de
$\Cond$ con estructura de grupo abeliano). Esta propiedad abre la puerta
a desarrollar álgebra homológica y cohomología en el marco condensado, lo cual es
uno de los grandes atractivos de este enfoque para la aritmética: permite
aplicar técnicas estándar de álgebra homológica a objetos que en la topología
clásica no forman una categoría abeliana.
\end{definition}

\subsubsection{La importancia de una categoría abeliana para la cohomología aritmética}\label{subsubsec:abeliano-aritmetica}

El problema que motiva las matemáticas condensadas en aritmética es
sencillo de enunciar: cuando intentamos aplicar herramientas estándar de álgebra
homológica a objetos aritméticos naturales (grupos de Galois, anillos $p$-ádicos,
variedades aritméticas), la topología clásica genera obstáculos técnicos
serios. La categoría de grupos topológicos no es abeliana, lo que impide
usar directamente la maquinaria estándar de cohomología de grupos, secuencias
exactas largas y derivación de funtores. Las cohomologías $p$-ádicas y de Galois
requieren definiciones ad hoc y cuidados técnicos constantes.

En el marco condensado, esta situación cambia radicalmente. La categoría de
grupos abelianos condensados $\Cond(\mathbf{Ab})$ es una categoría abeliana de
Grothendieck, lo que significa que:

\begin{itemize}
  \item Admite todas las construcciones estándar de álgebra homológica:
  núcleos, conúcleos, imágenes, coimágenes, secuencias exactas.
  \item Los funtores derivados (Ext, Tor) están bien definidos y se comportan
  como en categorías abelianas clásicas.
  \item La cohomología de grupos profinitos (como $G_K$ o $\Zp$) se vuelve
  simplemente la cohomología estándar de grupos abelianos condensados.
  \item Los productos tensoriales y completaciones son compatibles con la
  estructura condensada.
\end{itemize}

Esto permite reformular toda la teoría de cohomología $p$-ádica, representaciones
de Galois y geometría aritmética en un marco unificado donde las técnicas
estándar de álgebra homológica aplican directamente. La siguiente tabla resume
las diferencias clave:

\begin{table}[h]
  \centering
  \begin{tabular}{|l|p{5.5cm}|p{5.5cm}|}
    \hline
    \textbf{Aspecto} & \textbf{Topología clásica} & \textbf{Matemáticas condensadas} \\
    \hline
    Categoría de grupos & No es abeliana & Abeliana (Grothendieck) \\
    \hline
    Cohomología & Definiciones ad hoc & Máquina estándar \\
    \hline
    Objetos generadores & --- & Profinitos ($\Zp$, $G_K$, etc.) \\
    \hline
    Productos tensoriales & Problemas de compatibilidad & Bien comportados \\
    \hline
    Secuencias exactas & Limitadas & Completamente funcionales \\
    \hline
  \end{tabular}
  \caption{Comparación entre topología clásica y matemáticas condensadas en el
  contexto aritmético.}
  \label{tab:topologia-vs-condensado}
\end{table}

\begin{remark}\label{rem:condensados-interpretacion}
En términos intuitivos, un conjunto condensado es simplemente una forma de
decir que "conocemos un objeto por cómo lo ven los espacios de Stone". Para cada
profinito $S$, el conjunto $X(S)$ codifica la información que el objeto $X$
revela cuando es examinado mediante la sonda $S$. La condición de gavilla
garantiza que esta información es coherente: si cubrimos una sonda $S$ mediante
otras sondas $S_i$, la información sobre $X(S)$ puede reconstruirse a partir
de la información sobre los $X(S_i)$ y sus compatibilidades en las
intersecciones.
\end{remark}

\begin{figure}[h]
  \centering
  \begin{tikzcd}[row sep=large, column sep=large]
    X(S) \arrow[r] & \prod_{i\in I} X(S_i) \arrow[r, shift left, "p_1"] \arrow[r, shift right, "p_2"'] & \prod_{i,j\in I} X(S_i\times_S S_j)
  \end{tikzcd}
  \caption{Condición de gavilla para conjuntos condensados. El conjunto $X(S)$
  es el ecualizador del par de flechas $(p_1, p_2)$, es decir, consiste en
  aquellos elementos $(x_i)_{i\in I} \in \prod_{i\in I} X(S_i)$ que son
  compatibles: $p_1((x_i)) = p_2((x_i))$ en $\prod_{i,j\in I} X(S_i\times_S
  S_j)$.}
  \label{fig:sheaf-condition}
\end{figure}

La condición de gavilla expresa que $X(S)$ está determinado por sus valores
en los $S_i$ y las compatibilidades en las intersecciones $S_i\times_S S_j$.
Esta es la versión profinita de la condición clásica de gavilla en topología.

\begin{remark}\label{rem:condensados-vs-espacios}
La categoría $\Cond$ contiene de manera natural a la categoría de espacios
topológicos razonables (por ejemplo, compactos Hausdorff o espacios
completamente regulares). Si $T$ es un espacio topológico, podemos asociarle
su \emph{condensación} definiendo el funtor
\[
T^\diamondsuit:\ProFin^{\mathrm{op}}\longrightarrow\Set,\qquad
T^\diamondsuit(S):=\Cont(S,T),
\]
donde $\Cont(S,T)$ denota el conjunto de aplicaciones continuas $S\to T$.
\end{remark}

\begin{proposition}\label{prop:condensacion-gavilla}
Para todo espacio topológico compacto Hausdorff $T$, el funtor $T^\diamondsuit$
es una gavilla sobre $(\ProFin,J)$, es decir, $T^\diamondsuit\in\Cond$.
\end{proposition}

\begin{proof}
Sea $\{f_i:S_i\to S\}_{i\in I}$ un cubrimiento de $S$ en $(\ProFin,J)$.
Necesitamos verificar que el diagrama
\[
\Cont(S,T) \longrightarrow \prod_{i\in I} \Cont(S_i,T) \rightrightarrows \prod_{i,j\in I} \Cont(S_i\times_S S_j,T)
\]
es un ecualizador. 

Dado un sistema compatible $(g_i:S_i\to T)_{i\in I}$ (es decir, $g_i\circ
p_i = g_j\circ p_j$ en $S_i\times_S S_j$), necesitamos construir una única
$g:S\to T$ tal que $g\circ f_i = g_i$ para todo $i$. Como $\{f_i\}$ es un
cubrimiento epimórfico finito y $T$ es compacto Hausdorff, la existencia y
unicidad de $g$ se sigue de las propiedades estándar de espacios compactos
y de la compatibilidad del sistema $(g_i)$.
\end{proof}

Así, $T^\diamondsuit$ examina $T$ mediante todas
las sondas profinitas, reteniendo únicamente la información continua que es
estable respecto a límites inversos finitos. Esta construcción define un
funtor fiel
\[
(-)^\diamondsuit:\CHaus\longrightarrow\Cond,
\]
donde $\CHaus$ denota la categoría de espacios compactos Hausdorff.

\begin{figure}[h]
  \centering
  \begin{tikzpicture}[>=stealth, node distance=3.5cm, auto]
    \node (T) {$T$};
    \node (Tdiam) [right=of T] {$T^\diamondsuit$};
    
    \node (S1) [below=1.5cm of T] {$S_1$};
    \node (S2) [below=2.5cm of T] {$S_2$};
    \node (S3) [below=3.5cm of T] {$S_3$};
    \node (dots) [below=4.5cm of T] {$\vdots$};
    \node (S) [below=5.5cm of T] {$S \in \ProFin$};
    
    \draw[->] (S1) to node[left] {$\Cont(S_1,T)$} (T);
    \draw[->] (S2) to node[left] {$\Cont(S_2,T)$} (T);
    \draw[->] (S3) to node[left] {$\Cont(S_3,T)$} (T);
    \draw[->] (S) to node[left] {$\Cont(S,T)$} (T);
    
    \draw[->,dashed] (T) to node[above] {condensación} (Tdiam);
    
    \draw[->] (S1) to node[below] {$T^\diamondsuit(S_1)$} (Tdiam);
    \draw[->] (S2) to node[below] {$T^\diamondsuit(S_2)$} (Tdiam);
    \draw[->] (S3) to node[below] {$T^\diamondsuit(S_3)$} (Tdiam);
    \draw[->] (S) to node[below] {$T^\diamondsuit(S)$} (Tdiam);
    
    \node[right=0.5cm of Tdiam] {
      \begin{minipage}{4cm}
        \small
        Funtor:\\
        $\ProFin^{\mathrm{op}} \to \Set$\\
        $S \mapsto \Cont(S,T)$
      \end{minipage}
    };
  \end{tikzpicture}
  \caption{La condensación $T^\diamondsuit$ de un espacio topológico $T$ examina $T$ mediante todas las sondas profinitas $S \to T$. Cada profinito $S$ proporciona información sobre $T$ a través del conjunto $\Cont(S,T)$ de aplicaciones continuas.}
  \label{fig:condensation}
\end{figure}
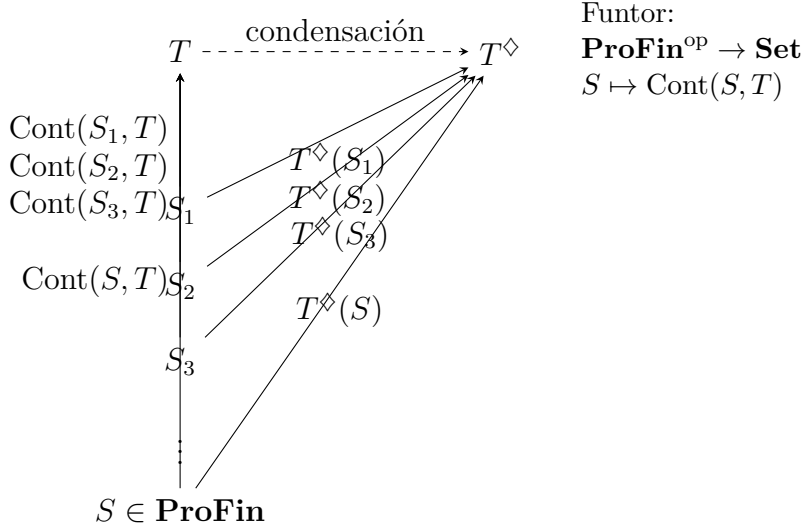

\begin{example}[Espacios discretos]\label{ex:condensado-discreto}
Sea $X$ un conjunto discreto (con topología discreta). Su condensación
$X^\diamondsuit$ satisface
\[
X^\diamondsuit(S):=\Cont(S,X)
\]
para todo profinito $S$. Como $S$ es compacto y totalmente desconectado, y
$X$ es discreto, toda aplicación continua $f:S\to X$ debe ser
\emph{localmente constante}: existe una partición finita de $S$ en clopens
$\{U_1,\ldots,U_n\}$ tal que $f$ es constante en cada $U_i$. En particular,
$f(S)$ es finito.

Más aún, como $S$ es profinito, podemos escribirlo como límite inverso de
espacios finitos $S=\varprojlim S_i$. Cualquier función continua $f:S\to X$
debe ser localmente constante (pues $S$ es compacto y totalmente desconectado,
y $X$ es discreto), y por tanto factoriza a través de algún $S_i$ finito. Esto
muestra que $\Cont(S,X)$ puede describirse en términos de funciones sobre los
$S_i$ finitos, reflejando cómo la condensación filtra automáticamente la
información: las sondas profinitas sólo ven la parte finita y discretamente
estratificada de un espacio discreto arbitrario, reflejando la estructura
profinita de $S$.
\end{example}

\begin{example}[Enteros $p$-ádicos como sondas]\label{ex:Zp-sonda}
En el Ejemplo~\ref{ex:Zp-profinito} construimos los enteros $p$-ádicos como
límite inverso
\[
\Zp=\varprojlim_{n\ge 1}\Z/p^n\Z,
\]
con aplicaciones de transición naturales $\varphi_{n+1,n}:\Z/p^{n+1}\Z\to
\Z/p^n\Z$. Por la Proposición~\ref{prop:profinite-stone}, $\Zp$ es un
espacio profinito, y por tanto puede usarse como sonda.

Si $T$ es un espacio compacto Hausdorff, su condensación verifica
\[
T^\diamondsuit(\Zp) = \Cont(\Zp,T).
\]

Ahora, como $\Zp=\varprojlim_{n\ge 1}\Z/p^n\Z$ y cada $\Z/p^n\Z$ es finito
(discreto), tenemos que
\[
\Cont(\Zp,T) \cong \varprojlim_{n\ge 1} \Cont(\Z/p^n\Z,T),
\]
donde el límite inverso se toma respecto a las aplicaciones inducidas por
$\varphi_{n+1,n}:\Z/p^{n+1}\Z\to\Z/p^n\Z$. Esto significa que una aplicación
continua $f:\Zp\to T$ queda determinada por un sistema coherente de funciones
$f_n:\Z/p^n\Z\to T$ tales que $f_n = f_{n+1}\circ\varphi_{n+1,n}$ para todo $n$.

Este hecho muestra cómo la teoría condensada reconstruye las
condiciones de compatibilidad que aparecen en la definición de límites
inversos aritméticos. La estructura profinita de $\Zp$ se refleja en cómo
$T^\diamondsuit$ lo percibe: como un sistema compatible de aproximaciones
finitas.
\end{example}

\begin{example}[Grupos de Galois como sondas]\label{ex:galois-sonda}
Sea $K$ un campo con grupo de Galois absoluto $G_K=\Gal(\overline{K}/K)$.
Por la Sección~\ref{subsec:galois}, $G_K$ es profinito:
\[
G_K \cong \varprojlim_{L/K \textit{ finita}} \Gal(L/K),
\]
donde $L$ recorre las extensiones finitas de $K$.

Si $T$ es un espacio compacto Hausdorff, entonces $T^\diamondsuit(G_K) =
\Cont(G_K,T)$ consiste en todas las aplicaciones continuas $G_K\to T$. Si
además $T$ tiene una acción continua de $G_K$, entonces podemos considerar el
subconjunto de aplicaciones $G_K$-equivariantes, que codifican información
sobre cómo $G_K$ actúa sobre $T$. Esta construcción permite estudiar espacios
con acción de Galois desde la perspectiva condensada, donde la estructura
profinita de $G_K$ se integra naturalmente en el marco de sondas.

Un ejemplo concreto: si consideramos $\Zp$ como espacio topológico con la
acción trivial de $G_K$, entonces $\Zp^\diamondsuit(G_K) = \Cont(G_K,\Zp)$
consiste en todas las aplicaciones continuas del grupo de Galois a los enteros
$p$-ádicos. Estas aplicaciones codifican información sobre cómo $G_K$ puede
"medirse" mediante $\Zp$, y son fundamentales en la teoría de representaciones
de Galois y en la cohomología $p$-ádica de variedades aritméticas.
\end{example}

\begin{example}[$\Zp$ examinándose a sí mismo]\label{ex:Zp-Zp}
Un ejemplo ilustrativo de cómo funciona la condensación es considerar
$\Zp^\diamondsuit(\Zp)$, es decir, cómo los enteros $p$-ádicos se ven a sí
mismos cuando son examinados mediante la sonda $\Zp$. Por definición,
\[
\Zp^\diamondsuit(\Zp) = \Cont(\Zp,\Zp),
\]
el conjunto de todas las aplicaciones continuas $\Zp\to\Zp$.

Como $\Zp$ es profinito y compacto, toda aplicación continua $f:\Zp\to\Zp$ está
determinada por su comportamiento en los niveles finitos. Más precisamente,
como $\Zp = \varprojlim_{n\ge 1} \Z/p^n\Z$, cada aplicación continua $f:\Zp\to\Zp$
induce un sistema compatible de aplicaciones $f_n:\Z/p^n\Z\to\Z/p^n\Z$ tales que
$f_n\circ \varphi_{n+1,n} = \varphi_{n+1,n}\circ f_{n+1}$ para todo $n$, donde
$\varphi_{n+1,n}:\Z/p^{n+1}\Z\to\Z/p^n\Z$ es la reducción módulo $p^n$.

Esto muestra que $\Cont(\Zp,\Zp)$ puede describirse en términos de sistemas
compatibles de aplicaciones sobre los niveles finitos, reflejando cómo la
estructura profinita de $\Zp$ se preserva bajo aplicaciones continuas.
Este ejemplo muestra cómo la condensación captura la estructura interna de
$\Zp$ mediante su propia estructura profinita, estableciendo un puente entre
la topología profinita y la estructura aritmética.
\end{example}

\begin{example}[La recta real condensada]\label{ex:R-condensado}
Consideremos el espacio topológico $\R$ con su topología usual. Su
condensación $\R^\diamondsuit$ asigna a cada profinito $S$ el conjunto
\[
\R^\diamondsuit(S) = \Cont(S,\R)
\]
de funciones continuas $S\to\R$. 

Aunque $\R$ no es compacto, su condensación $\R^\diamondsuit$ es un objeto
bien definido en $\Cond$. La clave es que, al examinar $\R$ mediante sondas
profinitas, capturamos únicamente la información "localmente compacta" o
"profinita" de $\R$. Por ejemplo, si $S$ es finito, entonces
$\Cont(S,\R)\cong\R^{|S|}$ es simplemente un producto cartesiano finito de
copias de $\R$.

Esta construcción ocupa un lugar central en la geometría analítica condensada de
Clausen--Scholze \cite{ClausenScholzeLecturesAnalytic}, donde espacios no compactos como $\R$ o $\C$ se estudian
mediante sus condensaciones, lo que abre la posibilidad de desarrollar análisis y geometría
en el marco condensado.
\end{example}

\subsection{Clopens como subobjetos representables}\label{subsec:clopens-subobjetos}

La dualidad de Stone que hemos desarrollado a lo largo de este artículo se
refleja de manera natural en la categoría condensada. Sea $S$ un espacio
profinito. Por el Teorema~\ref{thm:stone-representation}, el álgebra de
Boole $\Clop(S)$ está en dualidad con $S$ mismo. Esta dualidad se extiende
al marco condensado de la siguiente manera:

\begin{proposition}\label{prop:clopens-subobjetos}
Sea $S$ un espacio profinito y sea $U\subseteq S$ un clopen. Entonces la
inclusión $i:U\hookrightarrow S$ induce un monomorfismo
\[
U^\diamondsuit\hookrightarrow S^\diamondsuit
\]
en la categoría $\Cond$, donde $U^\diamondsuit$ y $S^\diamondsuit$ son las
condensaciones de $U$ y $S$ respectivamente.
\end{proposition}

\begin{proof}
Para cada profinito $T$, necesitamos verificar que la aplicación inducida
\[
\Cont(T,U)\longrightarrow\Cont(T,S)
\]
es inyectiva. Esto es inmediato: si $f,g:T\to U$ son continuas y
$i\circ f = i\circ g$, entonces $f=g$ porque $i$ es inyectiva.

Para ver que es un monomorfismo en $\Cond$, supongamos que tenemos dos
morfismos $X\rightrightarrows U^\diamondsuit$ en $\Cond$ que se igualan al
componer con $U^\diamondsuit\hookrightarrow S^\diamondsuit$. Como los
profinitos generan $\Cond$ (Teorema~\ref{thm:profinites-generadores}), basta
verificar que los morfismos coinciden al evaluar en todos los profinitos. Para
cualquier profinito $T$, la composición con $U^\diamondsuit\hookrightarrow
S^\diamondsuit$ da morfismos $X(T)\rightrightarrows\Cont(T,S)$ que son iguales
por hipótesis. Como la inclusión $U\hookrightarrow S$ es inyectiva, esto
implica que los morfismos originales $X(T)\rightrightarrows\Cont(T,U)$ son
iguales. Por tanto, los morfismos son iguales como transformaciones naturales.
\end{proof}

Esta construcción establece un morfismo inyectivo
\[
\Clop(S)\hookrightarrow\Sub_{\Cond}(S^\diamondsuit),
\]
donde $\Sub_{\Cond}(S^\diamondsuit)$ denota el conjunto de subobjetos de
$S^\diamondsuit$ en $\Cond$ (es decir, clases de equivalencia de
monomorfismos $X\hookrightarrow S^\diamondsuit$).

\begin{remark}\label{rem:boole-condensado}
Así, la estructura booleana de clopens se traduce directamente en
la lógica interna de subobjetos representables en $\Cond$, haciendo visible
la álgebra de Boole como parte del andamiaje categórico. Las
operaciones booleanas (unión, intersección, complemento) se reflejan en
operaciones sobre subobjetos en $\Cond$, estableciendo un puente entre la
lógica booleana clásica y la lógica interna de la categoría condensada.

Esta conexión muestra cómo la dualidad
de Stone, que comenzamos estudiando como una correspondencia
álgebra--topología, se extiende naturalmente al marco condensado, donde
adquiere un significado categórico más profundo.
\end{remark}

\subsection{Profinitos como generadores de la categoría condensada}

Lo que justifica todo el enfoque de Clausen--Scholze es que
los profinitos son suficientes para \emph{generar} la categoría $\Cond$ en un
sentido técnico preciso. Esta propiedad convierte a $\ProFin$ en la
base lógica y geométrica sobre la que se construye toda la estructura
condensada.

\begin{theorem}[Profinitos como generadores]\label{thm:profinites-generadores}
La categoría $\Cond$ está generada por los objetos de $\ProFin$ en el sentido
siguiente:
\begin{enumerate}[label=\roman*)]
  \item \emph{Separación}: Si $f,g:X\to Y$ son morfismos en $\Cond$ y
  $f(S)=g(S):X(S)\to Y(S)$ para todo $S\in\ProFin$, entonces $f=g$.
  
  \item \emph{Detección de isomorfismos}: Un morfismo $f:X\to Y$ en $\Cond$ es
  un isomorfismo si y sólo si $f(S):X(S)\to Y(S)$ es una biyección para todo
  $S\in\ProFin$.
  
  \item \emph{Presentación}: Todo objeto $X\in\Cond$ puede presentarse como
  colímite de un diagrama cuyos objetos son de la forma $S^\diamondsuit$ para
  $S\in\ProFin$. Más precisamente, $X$ es el colímite de su restricción a la
  subcategoría plena de profinitos, es decir, $X \cong \varinjlim_{S\to X} S^\diamondsuit$,
  donde el límite se toma sobre el diagrama de todas las sondas profinitas $S\to X$.
\end{enumerate}
\end{theorem}

\begin{proof}[Esbozo]
(i) y (ii) se siguen de que los profinitos forman una clase generadora en
$\Cond$: dos morfismos son iguales si y sólo si coinciden al evaluar en todos
los objetos generadores. Esto es una consecuencia directa de la definición de
topos de Grothendieck y del hecho de que $\ProFin$ es una subcategoría densa
de $\Cond$.

Para (iii), la idea clave es que cualquier conjunto condensado $X$ puede
expresarse como el colímite de todos sus valores en profinitos. Más
precisamente, consideramos la categoría de elementos de $X$ restringida a
$\ProFin$: los objetos son pares $(S,x)$ donde $S\in\ProFin$ y $x\in X(S)$, y
los morfismos $(S,x)\to (T,y)$ son morfismos $f:S\to T$ en $\ProFin$ tales que
$X(f)(y)=x$. El colímite de los funtores representables $S^\diamondsuit$ sobre
esta categoría es isomorfo a $X$. Esta construcción es análoga a la
presentación estándar de una gavilla como colímite de representables en teoría
de topos.
\end{proof}

\begin{remark}\label{rem:presentacion-colimite-detalle}
Una presentación más detallada de esta construcción (mostrando explícitamente
cómo $X$ se presenta como colímite de representables profinitos) puede
encontrarse en las notas de Clausen--Scholze \cite{ScholzeCondensed2} (Sección~2)
o en el survey de Bhatt \cite{Bhatt2022}. Aquí nos limitamos a enunciar el
resultado y sus consecuencias inmediatas para la detección de morfismos y
propiedades en $\Cond$, que son suficientes para entender el papel fundamental de
los profinitos como generadores de la categoría condensada.
\end{remark}

Esta propiedad tiene consecuencias profundas:

\begin{corollary}\label{cor:continuidad-profinita}
Toda noción de continuidad, morfismo o propiedad en $\Cond$ puede detectarse
únicamente mediante evaluaciones en objetos profinitos. En particular, dos
espacios topológicos compactos Hausdorff $T$ y $T'$ tienen condensaciones
isomorfas $T^\diamondsuit \cong T'^\diamondsuit$ si y sólo si existe una
biyección natural entre $\Cont(S,T)$ y $\Cont(S,T')$ para todo profinito $S$,
es decir, si los funtores $T^\diamondsuit$ y $T'^\diamondsuit$ son naturalmente
isomorfos.
\end{corollary}

\begin{example}[Detección de continuidad]\label{ex:deteccion-continuidad}
Supongamos que tenemos dos aplicaciones $f,g:T\to T'$ entre espacios
compactos Hausdorff. Entonces $f=g$ si y sólo si, para todo profinito $S$ y
toda aplicación continua $h:S\to T$, se tiene $f\circ h = g\circ h$. Esto
muestra que la continuidad en el sentido condensado se detecta completamente
mediante sondas profinitas.
\end{example}

\subsection{Ventajas del enfoque condensado para la aritmética}

El marco condensado resuelve los problemas mencionados en la
Introducción~\ref{sec:intro} de la siguiente manera:

\begin{enumerate}
  \item \emph{Categorías abelianas}: La categoría de grupos condensados (o
  módulos condensados) es abeliana, a diferencia de la categoría de grupos
  topológicos. Esto permite usar toda la maquinaria de álgebra homológica
  estándar.
  
  \item \emph{Cohomologías naturales}: Las cohomologías en el marco condensado
  se definen de manera natural usando la estructura abeliana, sin necesidad de
  definiciones ad hoc.
  
  \item \emph{Productos tensoriales}: Los productos tensoriales de módulos
  condensados son compatibles con la estructura topológica de manera natural,
  resolviendo los problemas de completación que aparecen en el marco clásico.
\end{enumerate}

Estas ventajas son posibles porque los profinitos, que ya
aparecen naturalmente en aritmética, juegan el papel de objetos generadores
(Teorema~\ref{thm:profinites-generadores}). Así, la estructura aritmética se integra
directamente en el andamiaje categórico.

\begin{remark}\label{rem:profinites-aritmetica}
Este es el punto donde la metáfora del título adquiere su significado más
profundo: los objetos aritméticos (las "piedras") están "envueltos" por el
mar de la estructura condensada, que está generada por los profinitos. Cada
sonda profinita $S\to X$ es como una ola que toca la piedra, y la
información que recoge (codificada en $X^\diamondsuit(S)$) es la única
información que la estructura condensada retiene. Los profinitos son las
"ondas" que permiten que el mar "envuelva" completamente las piedras
aritméticas.
\end{remark}

Los conjuntos condensados proporcionan la base para desarrollar
tanto la teoría categórica general como la geometría analítica condensada,
tal como se expone en detalle en \cite{ClausenScholzeLecturesAnalytic}. Esta
teoría permite formular y resolver problemas aritméticos que serían
inaccesibles en el marco topológico clásico.

\subsection{Hacia un panorama más amplio}

Como se mencionó en la Introducción, este artículo es el primero de una serie de
artículos expositivos dedicados a explorar las relaciones entre álgebra y topología.
El proyecto completo tiene como objetivo reorganizar, desde una perspectiva
aritmética y categórica, la red de correspondencias entre estructuras algebraicas,
lógicas y topológicas. Las dualidades discutidas aquí (Boole--Stone y ProFin) ocupan
la primera capa del esquema. Los artículos siguientes de esta serie incluirán:

\begin{itemize}
    \item la \textbf{dualidad de Priestley} para retículos distributivos;
    \item la \textbf{dualidad de Esakia} para álgebras de Heyting y su conexión con la lógica intuicionista;
    \item la teoría de \textbf{marcos} y \textbf{locales} como generalización topológica sin puntos;
    \item la descripción categórica de \textbf{topoi elementales} como unificación natural de estas interpretaciones;
    \item la reinterpretación de estas capas dentro del marco condensado.
\end{itemize}

La intención de este programa es situar estas estructuras clásicas en un
paisaje coherente, donde las dualidades reticulares, booleanas, intuicionistas
y locales puedan comprenderse desde un punto de vista aritmético y compatible
con la estructura profinita que subyace a la teoría condensada.

Con esta perspectiva, el presente artículo debe entenderse no como un tratado
completo, sino como el punto de partida necesario: la base booleana y profinita
sobre la cual se construyen los niveles más ricos de la arquitectura lógica y
geométrica moderna.

\section{Conclusión y ejercicios sugeridos}\label{sec:conclusion}

Hemos recorrido el arco conceptual que conecta álgebras de Boole, espacios de Stone,
límites inversos profinitos y la compactificación de Stone--Čech, culminando en la
perspectiva condensada de Clausen--Scholze. El punto central, visto desde la
perspectiva aritmética que hemos enfatizado a lo largo del artículo, es que la
dualidad
\[
\Bool \;\simeq\; \Stone^{\mathrm{op}} \;\simeq\; \ProFin^{\mathrm{op}}
\]
no es un artificio categórico, sino un ejemplo concreto de cómo estructuras
algebraicas discretas y espacios totalmente desconectados se reflejan mutuamente
a través de equivalencias de categorías. Esta dualidad es especialmente relevante
en aritmética porque muchos objetos fundamentales (enteros $p$-ádicos, grupos de
Galois, completaciones profinitas) aparecen naturalmente como espacios profinitos,
y su estructura lógica se codifica en álgebras de Boole.

La compactificación de Stone--Čech proporciona la envoltura
booleana universal de un conjunto discreto; la completación profinita recoge la
información finita detectada por todos los cocientes; y la condensación
$X^\diamondsuit$ reorganiza la información continua mediante sondas profinitas,
poniendo de manifiesto la capa lógica que subyace a la noción de continuidad en
el sentido condensado. Desde una perspectiva aritmética, estas tres
construcciones pueden interpretarse como estratos complementarios a través de
los cuales un mismo objeto discreto se analiza: el punto esencial no es elegir
uno de ellos, sino entender cómo interactúan. Por ejemplo, cuando estudiamos
$\mathbb{Z}$, podemos examinarlo a través de $\beta\mathbb{Z}$ (compactificación
booleana), $\widehat{\mathbb{Z}}$ (completación profinita) o $\mathbb{Z}^\diamondsuit$
(condensación), cada una revelando aspectos diferentes de su estructura aritmética.

La conexión con las Matemáticas Condensadas (Sección~\ref{sec:condensed}) muestra
cómo estos conceptos clásicos adquieren un nuevo significado en el marco
categórico moderno. Los profinitos, que aparecen naturalmente en aritmética (como
vimos en la Sección~\ref{sec:ejemplos-aritmeticos} con los enteros $p$-ádicos y
grupos de Galois), juegan el papel de generadores de la categoría
condensada, lo que hace posible formular y resolver problemas que serían inaccesibles en el
marco topológico clásico. Esta perspectiva es especialmente valiosa en geometría
aritmética, donde posibilita desarrollar teorías de cohomología, productos tensoriales
y completaciones que son compatibles con la estructura aritmética subyacente. La
dualidad de Stone se extiende naturalmente a este contexto, estableciendo un
puente entre la lógica booleana clásica y la estructura categórica contemporánea.

Retomando la metáfora que ha guiado este artículo: en la matemática condensada, no
solo el mar (la estructura condensada) envuelve a las piedras (los objetos
aritméticos), sino que las piedras profinitas definen la salinidad y la dinámica
de ese mar. Los profinitos no son meros objetos sumergidos en el mar condensado,
sino los elementos que determinan sus propiedades esenciales. Cada sonda profinita
$S\to X$ no solo toca la piedra, sino que contribuye a definir la naturaleza
misma del medio que la envuelve. La dualidad de Stone y la
estructura profinita no son solo herramientas para estudiar objetos aritméticos,
sino los principios organizadores que dan forma al mar mismo de la matemática
condensada.

Este artículo constituye el primer artículo de una serie expositiva dedicada a
explorar las relaciones entre álgebra y topología. Los artículos siguientes de esta
serie incluirán dualidades más generales (Priestley para retículos distributivos,
Esakia para álgebras de Heyting), teoría de marcos y locales, y topoi elementales,
todas reinterpretadas desde una perspectiva aritmética y compatible con la estructura
profinita que subyace a la teoría condensada.

\vspace{0.6cm}

\noindent\textbf{Ejercicios complementarios.}  
Los ejercicios siguientes permiten comprobar de forma directa varias de las
construcciones discutidas en el texto. No introducen material adicional, pero
recogen equivalencias y compatibilidades que aparecen de manera natural en la
teoría.

\begin{exercise}\label{ex:Zp-Cantor}
Demuestra que $\Zp$ es homeomorfo a un conjunto de Cantor generalizado. En el
caso $p=2$, escribe cada elemento de $\mathbb{Z}_2$ como serie
$\sum_{n\geq 0} a_n 2^n$ con $a_n\in\{0,1\}$, y compara con el espacio
$2^{\mathbb{N}}$ con la topología producto.
\end{exercise}

\begin{exercise}\label{ex:betaN-ex}
Recupera la compactificación de Stone--Čech $\beta\mathbb{N}$ como espacio de
Stone de la álgebra de Boole $\mathcal{P}(\mathbb{N})$, y verifica directamente
su propiedad universal (Teorema~\ref{thm:betaX-universal-discreto}). Describe
explícitamente la aplicación $\beta\mathbb{N}\to K$ para un compacto Hausdorff
$K$ dado y una función $f:\mathbb{N}\to K$.
\end{exercise}

\begin{exercise}\label{ex:betaZ-hatZ}
En el ejemplo de $\beta\mathbb{Z}$ (Ejemplo~\ref{ex:betaZ-profinita}), muestra
que la aplicación $\pi:\beta\mathbb{Z}\to \widehat{\mathbb{Z}}$ construida a
partir de los $r_n$ es continua, suprayectiva y respeta la estructura de grupo.
¿Es un isomorfismo de grupos topológicos? ¿Por qué?
\end{exercise}

\begin{exercise}\label{ex:betaX-funciones}
Sea $X$ un espacio discreto y sea $C(X)$ el anillo de funciones acotadas
$f:X\to\C$, con la norma del supremo. Sea $C(\beta X)$ el anillo de funciones
continuas $\beta X\to\C$, también con la norma del supremo. Muestra que $C(X)$
es isométricamente isomorfo a $C(\beta X)$ como álgebras normadas.

\emph{Pista:} usa la propiedad universal de $\beta X$ y la compacidad de
$\beta X$ para extender de manera única cada función acotada $f:X\to\C$ a una
función continua $\tilde f:\beta X\to\C$, y verifica que la extensión preserva
la norma.
\end{exercise}

\begin{exercise}[Sondas profinitas]\label{ex:profinet-probes}
Sean $T$ y $T'$ espacios compactos Hausdorff, y sean $f,g:T\to T'$ aplicaciones
continuas. Usa la propiedad universal de la condensación (Proposición~\ref{prop:condensacion-gavilla})
y el hecho de que los profinitos generan $\Cond$ (Teorema~\ref{thm:profinites-generadores})
para demostrar que $f = g$ si y sólo si
\[
f\circ h = g\circ h
\]
para toda aplicación continua $h:S\to T$ con $S$ profinito.

\emph{Pista:} considera las condensaciones $T^\diamondsuit$ y $T'^\diamondsuit$,
y observa que $f$ y $g$ inducen morfismos $f^\diamondsuit, g^\diamondsuit:
T^\diamondsuit\to T'^\diamondsuit$ en $\Cond$. Usa que los profinitos detectan
morfismos en $\Cond$ para concluir.

Este ejercicio muestra que los profinitos detectan la continuidad en $\Cond$ y
pone de manifiesto que $\ProFin$ genera la categoría condensada en el sentido
usual.
\end{exercise}


\end{document}